\documentclass[11pt,reqno]{amsart}
\usepackage{amsmath}
\usepackage{amssymb}
\usepackage{amsthm}
\usepackage{graphicx}
\usepackage{url}
\usepackage{tikz}
\usepackage[utf8]{inputenc}
\usepackage{cite}
\usepackage{xcolor}
\usepackage{tikz-cd}
\usepackage{hyperref}
\newcounter{mmacnt}
\def\restartmma{\setcounter{mmacnt}{0}}
\restartmma \catcode`|=\active
\def|#1|{\mathrm{#1}}
\catcode`|=12
\newenvironment{mma}{
  \par
 \catcode`|=\active
 \parskip=2pt\parindent=0pt 
 \small
 \def\In##1\\{%
   \def\linebreak{\hfill\break\null\qquad}%
   \refstepcounter{mmacnt}
   \hangindent=2.5em\hangafter=0
   \leavevmode
   \llap{\tiny\sffamily In[\arabic{mmacnt}]:=\kern.5em}%
   \mathversion{bold}\scriptsize$\tt\bf\displaystyle##1$\normalsize
   \mathversion{normal}\par
 }%
 \def\Print##1\\{%
   \def\linebreak{\hfill\break}%
   \hangindent=2.5em\hangafter=0
   \leavevmode\scriptsize ##1\par}%
 \def\Out##1\\{%
   \vspace*{-0.2cm}\def\linebreak{$\hfill\break\null\hfill$}%
   \kern\abovedisplayskip\par
   \hangindent=2.5em\hangafter=0
   \leavevmode
   \llap{\tiny\sffamily Out[\arabic{mmacnt}]=\kern.5em}
   \scriptsize$\displaystyle\tt##1$\normalsize\hfill\null\par
   \kern\belowdisplayskip\vspace*{-0.3cm}
 }%
 \def\Warning##1##2\\{%
   \def\linebreak{\hfill\break}%
   \hangindent=2.5em\hangafter=0
   \leavevmode
   {\scriptsize##1 : ##2}\par}%
}{%
 \par\smallskip
}

\newcommand{\LoadP}[1]{\fcolorbox{black}{white}{
\begin{minipage}[t]{9.5cm}
\footnotesize #1
\end{minipage}}}

\newcommand{\myOut}[1]{{\sffamily Out[#1]}}

\def\MLabel#1{{\refstepcounter{mmacnt}\label{#1}}\addtocounter{mmacnt}{-1}}

\usetikzlibrary{decorations.pathreplacing}

\renewcommand{\(}{\left\(}
\renewcommand{\)}{\right\)}
\renewcommand{\[}{\left\[}
\renewcommand{\]}{\right\]}

\numberwithin{equation}{section}

\theoremstyle{plain}

\newtheorem{theorem}{Theorem}[section]
\newtheorem{lemma}[theorem]{Lemma}

\newtheorem{corollary}[theorem]{Corollary}

\newtheorem{definition}[theorem]{Definition}

\newtheorem{question}[theorem]{Question}

\newtheorem{remark}[theorem]{Remark}

\makeatletter
\def\proof{\@ifnextchar[{\@oproof}{\@nproof}}
\def\@oproof[#1][#2]{\trivlist\item[\hskip\labelsep\textit{#2 Proof of\
		#1.}~]\ignorespaces}
\def\@nproof{\trivlist\item[\hskip\labelsep\textit{Proof.}~]\ignorespaces}

\makeatother

\begin{document}

\vspace*{-1.4cm}

\begin{flushleft} 
\hfill RISC Report Series 22--13
\end{flushleft}

\vspace*{1.4cm}

\title{Error bounds for the asymptotic expansion of the partition function}

%
%
\author{Koustav Banerjee}
\address{Research Institute for Symbolic Computation (RISC), Johannes Kepler University Linz, Altenberger Stra{\ss}e 69, 4040 Linz, Austria}
\email{koustav.banerjee@risc.jku.at}

\author{Peter Paule}
\address{Research Institute for Symbolic Computation (RISC), Johannes Kepler University Linz, Altenberger Stra{\ss}e 69, 4040 Linz, Austria}
\email{peter.paule@risc.jku.at}

\author{Cristian-Silviu Radu}
\address{Research Institute for Symbolic Computation (RISC), Johannes Kepler University Linz, Altenberger Stra{\ss}e 69, 4040 Linz, Austria}
\email{silviu.radu@risc.jku.at}

\author{Carsten Schneider}
\address{Research Institute for Symbolic Computation (RISC), Johannes Kepler University Linz, Altenberger Stra{\ss}e 69, 4040 Linz, Austria}
\email{Carsten.Schneider@risc.jku.at}




\begin{abstract}
Asymptotic study on the partition function $p(n)$ began with the work of Hardy and Ramanujan. Later Rademacher obtained a convergent series for $p(n)$ and an error bound was given by Lehmer. Despite having this, a full asymptotic expansion for $p(n)$ with an explicit error bound is not known. Recently O'Sullivan studied the asymptotic expansion of $p^{k}(n)$-partitions into $k$th powers, initiated by Wright, and consequently obtained an asymptotic expansion for $p(n)$ along with a concise description of the coefficients involved in the expansion but without any estimation of the error term. Here we consider a detailed and comprehensive analysis on an estimation of the error term obtained by truncating the asymptotic expansion for $p(n)$ at any positive integer $n$. This gives rise to an infinite family of inequalities for $p(n)$ which finally answers to a question proposed by Chen. Our error term estimation predominantly relies on applications of algorithmic methods from symbolic summation.    
\end{abstract}

\maketitle              

\tableofcontents



\section{Introduction}\label{sec:intro}
A partition of a positive integer $n$ is a non-increasing sequence of positive integers which sum to $n$, and the partition function $p(n)$ counts the number of partitions of $n$. In their epoch-making breakthrough work in the theory of partitions, Hardy and Ramanujan \cite{hardy1918asymptotic} proved that
\begin{equation}\label{HardyRamanujan0}
p(n) \sim \dfrac{1}{4n\sqrt{3}}e^{\pi\sqrt{\frac{2n}{3}}}\ \ \text{as}\ n \rightarrow \infty.
\end{equation}
They also proved that $p(n)$ is the integer nearest to 
\begin{equation}\label{HardyRamanujan1}
\dfrac{1}{2\sqrt{2}}\sum_{q=1}^{\nu}\sqrt{q}A_{q}(n)\psi_{q}(n),
\end{equation}
where $A_q(n)$ is a certain exponential sum, $\nu=\nu(n)$ is of the order of $\sqrt{n}$, and 
\begin{equation*}
\psi_q(n)=\dfrac{d}{dn}\Biggl(\exp\Bigl\{{\dfrac{C}{q}\lambda_n}\Bigr\}\Biggr),\ \lambda_n=\sqrt{n-\frac{1}{24}},\ C=\pi\sqrt{\frac{2n}{3}}.
\end{equation*}
 Extending $\nu$ to infinity, Lehmer \cite{lehmer1937hardy} proved that \eqref{HardyRamanujan1} is a divergent series. Rademacher \cite{rademacher1937convergent,rademacher1938partition,rademacher1943expansion} considered a modification of \eqref{HardyRamanujan1} that
presents a convergent series for $p(n)$ which reads:
\begin{equation}\label{Rademacher}
p(n)= \dfrac{1}{\pi \sqrt{2}}\sum_{k=1}^{\infty}A_k(n)\dfrac{d}{dn}\Biggl( \dfrac{\sinh\Bigl(C\lambda_n/k\Bigr)}{\lambda_n}\Biggr).
\end{equation}
Lehmer \cite{lehmer1938series,lehmer1939remainders} obtained an error bound after subtraction of the $N$th partial sum from the convergent series \eqref{Rademacher}. 
 
The study of a full asymptotic expansion for $p(n)$ can be traced in two directions by considering two different classes that arise from imposing restrictions on parts of partitions. The two restricted families are $p^{s}(n)$, the number of partitions of $n$ into perfect $s$th powers, and $p(n,k)$, the number of partitions of $n$ into at most $k$ parts. As an application of the ``circle method", Hardy and Ramanujan \cite[Section 7, 7.3]{hardy1918asymptotic} obtained the main term in the asymptotic expansion of $p^{s}(n)$. This of course retrieves \eqref{HardyRamanujan0} when we take $s=1$. Wright \cite{wright1934asymptotic,wright1935asymptotic} extended the work of Hardy and Ramanujan and obtained a full asymptotic expansion for $p^{s}(n)$. Recently O'Sullivan \cite{o2022detailed} proposed a simplified proof of Wright's results on the asymptotic expansion of $p^{s}(n)$, and consequently obtained an asymptotic formula for $p(n)$.
\begin{theorem}\cite[Proposition 4.4]{o2022detailed}\label{Sullivanthm0}
	Let $n$ and $R$ be positive integers. As $n \rightarrow \infty$,
\begin{equation}\label{Sullivan1}
p(n)=\dfrac{e^{\pi\sqrt{2n/3}}}{4n\sqrt{3}}\Biggl(1+\sum_{t=1}^{R-1}\dfrac{\omega_t}{\sqrt{n}^t}+O\Bigl(n^{-R/2}\Bigr)\Biggr),\end{equation}
with an implied constant depending only on $R$, where
\begin{equation}\label{Sullivan2}
\omega_t=\dfrac{1}{(-4\sqrt{6})^t}\sum_{k=0}^{\frac{t+1}{2}}\binom{t+1}{k}\dfrac{t+1-k}{(t+1-2k)!}\Bigl(\frac{\pi}{6}\Bigr)^{t-2k}.
\end{equation}
\end{theorem}
The binomial coefficient is defined as $\binom{x}{k}:=x(x-1)\dots(x-k+1)/k!$ if $k\in\mathbb{Z}_{\geq 0}$, $\binom{x}{0}:=1$, and $\binom{x}{k}:=0$ if $k\in\mathbb{Z}_{<0}$. 
Szekeres \cite{szekeres1953some} proposed an asymptotic expansion for $p(n,k)$ for $n$ and $k$ sufficiently large and considering $k=n$, one obtains the expansion for $p(n)$ as $p(n,k)=p(n)$. Canfield \cite{canfield1997recursions} proved Szekeres' result by using a recursion satisfied by $p(n,k)$ without using theory of complex functions and as a corollary, obtained the main term of the Hardy-Ramanujan formulas for $p(n)$, see \eqref{HardyRamanujan0}. For a probabilistic approach to the asymptotic expansion of $p(n)$, we refer to \cite{brassesco2020expansion}.

The primary objective of this paper is to obtain an explicit and computable error bound for the asymptotic expansion of $p(n)$. A main motivation to consider such a problem is that from the literature, including the works \cite{wright1934asymptotic,szekeres1953some,canfield1997recursions,brassesco2020expansion,o2022detailed}, we could not retrieve any information on the error bound for asymptotic expansion of $p(n)$. An advantage of getting a control over the error bound is that one can prove the $\log$-concavity property of $p(n)$ directly from the asymptotic expansion as speculated by Chen \cite[p. 121]{chen2010recent}. In the language of Theorem \ref{Sullivan1}, Chen's question can be formulated as follows: 
\begin{question}\label{Chenproblem}
Do there exists $d$ and $n_0$ such that
\begin{equation}\label{Chen0}
\dfrac{e^{\pi\sqrt{2n/3}}}{4n\sqrt{3}}\Biggl(1+\sum_{t=1}^{3}\dfrac{\omega_t}{\sqrt{n}^t}-\dfrac{d}{n^2}\Biggr)<p(n)<\dfrac{e^{\pi\sqrt{2n/3}}}{4n\sqrt{3}}\Biggl(1+\sum_{t=1}^{3}\dfrac{\omega_t}{\sqrt{n}^t}+\dfrac{d}{n^2}\Biggr)
\end{equation}
holds for all $n>n_0$?
\end{question}
Chen remarked that \eqref{Chen0} implies that $p(n)$ is $\log$-concave for sufficiently large $n$. Now in order to demystify the phrase ``sufficiently large", explicit information about $n_0$ is required; a question being  intricately connected with the computation of the error bound $d$. A similar phenomena can be found in O'Sullivan's work:
\begin{theorem}\cite[Theorem 1.3, (1.15)]{o2022detailed}\label{Sullivanthm1}
For each positive integer $k$ there exists $\mathcal{D}_k$ so that for all $n \geq \mathcal{D}_k$,
\begin{equation}\label{Sullivan3}
p^k(n)^2\geq p^k(n+1)\cdot p^k(n-1)\cdot (1+n^{-2}).
\end{equation}
\end{theorem}
For $k=1$, Theorem \ref{Sullivanthm1} merely implies that $p(n)$ is $\log$-concave for sufficiently large $n$ although we know that $(p(n))_{n \geq 26}$ is $\log$-concave due to \cite{nicolas1978entiers,desalvo2015log}. Moreover, O'Sullivan \cite[(5.17)]{o2022detailed} proved that for large enough $n$,
\begin{equation*}
\dfrac{p(n+1)p(n-1)}{p(n)^2}\Biggl(1+\dfrac{\pi}{\sqrt{24}n^{3/2}}\Biggr)>1,
\end{equation*}
 settled by Chen, Wang, and Xie \cite{chen2016finite}. The first three authors and Zeng \cite[Theorem 7.6]{BPRZ} proved a stronger version of \eqref{Sullivan3} using an infinite familiy of inequalities for $\log p(n)$.

We conclude this section by discussing the novelty of this paper in brevity. In order to elucidate the term $O\Bigl(n^{-R/2}\Bigr)$ in \eqref{Sullivan1}, determination of the asymptotic growth of the coefficients $\omega_t$ in \eqref{Sullivan2} is required; a task which looks deceptively simple. Our representation of $\omega_t$ is of the following form: 
$$\omega_t=\sum_{u=0}^{t}\gamma(u)\sum_{s=0}^{u}\psi(s).$$
In an effort to estimate the inner sum $\sum_{s=0}^{u}\psi(s)$, the use of the symbolic summation tool \texttt{Sigma} \cite{Schneider:07a} was essential. Schneider considered \cite{Schneider:07a,Schneider:13a,Schneider:21} a broader algorithmic framework that subsumes the theory of difference field and ring extensions together with the method of creative telescoping. This algorithmic tool began to be aimed at a wider class of multi-sums, most frequently encountered in problems of enumerative combinatorics. For example, in Andrews, Paule, and Schneider \cite{APS:05} we can see how \texttt{Sigma} assists to solve the TSPP-problem in an LU-reformulation by Andrews. Beyond the world of combinatorics, applications of \texttt{Sigma} transcends to solve a very general class of Feynman integrals which are of relevance for manifold physical processes in quantum field theory, see \cite{Particle:20}. This paper adds a new facet to the regime of applications of \texttt{Sigma}; in particular, its foray into asymptotic estimation for  partition-like functions seems to begin with this work.

\section{A roadmap for the reader}
In this section we will provide a roadmap on the structure of this paper; i.e., a navigation from the starting point to the final goal of this paper, to facilitate for the reader to follow.

Using the Hardy-Ramanujan-Rademacher formula for $p(n)$ and Lehmer's error bound, Chen, Jia, and Wang \cite[Lemma 2.2]{chen2019higher} proved that for all $n \geq 1207$,
\begin{equation}\label{ChenJiaWang1}
\dfrac{\sqrt{12}e^{\mu(n)}}{24n-1}\Biggl(1-\dfrac{1}{\mu(n)}-\dfrac{1}{\mu(n)^{10}}\Biggr)<p(n)<\dfrac{\sqrt{12}e^{\mu(n)}}{24n-1}\Biggl(1-\dfrac{1}{\mu(n)}+\dfrac{1}{\mu(n)^{10}}\Biggr),
\end{equation}
where for $n\geq 1$, $\mu(n):=\frac{\pi}{6}\sqrt{24n-1}$; a definition which is kept throughout this paper. More generally, due to the first three authors and Zeng, we have the following result.
\begin{theorem}\cite[Theorem 4.4]{BPRZ}\label{bprz1}
	For $k \in \mathbb{Z}_{\geq 2}$, define
	$$\widehat{g}(k):=\dfrac{1}{24}\Biggl(\dfrac{36}{\pi^2}\cdot \nu(k)^2+1\Biggr),$$
	where $\nu(k):=2\log 6+(2\log 2)k+2k \log k+2k \log \log k+\dfrac{5k \log \log k}{\log k}$. Then for all $k\in\mathbb{Z}_{\geq 2}$ and  $n> \widehat{g}(k)$ such that $(n,k)\neq (6,2)$, we have
	\begin{equation}\label{bprzeqn1}
	\frac{\sqrt{12}e^{\mu(n)}}{24n-1}\Biggl(1-\frac{1}{\mu(n)}-\frac{1}{\mu(n)^k}\Biggr)<p(n)<\frac{\sqrt{12}e^{\mu(n)}}{24n-1}\Biggl(1-\frac{1}{\mu(n)}+\frac{1}{\mu(n)^k}\Biggr).
	\end{equation}
\end{theorem}
The goal of this paper is to derive an inequality of the form
\begin{equation}\label{goaleqn1}
\frac{e^{\pi\sqrt{2n/3}}}{4n\sqrt{3}}\Biggl(\sum_{t=0}^{k-1}\frac{g(t)}{\sqrt{n}^j}+\frac{L(k)}{\sqrt{n}^k}\Biggr) < p(n) < \frac{e^{\pi\sqrt{2n/3}}}{4n\sqrt{3}}\Biggl(\sum_{t=0}^{k-1}\frac{g(t)}{\sqrt{n}^t}+\frac{U(k)}{\sqrt{n}^k}\Biggr),
\end{equation}
stated precisely in Theorem \ref{Mainthm}, starting from the inequality \eqref{bprzeqn1}. As a consequence we obtain Corollary \ref{cor1} which will give an explicit answer to the problem stated in Question \ref{Chenproblem} and which, as a further consequence reveals that $p(n)$ is $\log$-concave for all $n \geq 26$, see Remark \ref{corremark1}.

The first step is to find explicitly the coefficients $g(t)$ such that
$$\dfrac{\sqrt{12}\ e^{\mu(n)}}{24n-1}\Bigl(1-\dfrac{1}{\mu(n)}\Bigr)=\frac{e^{\pi\sqrt{2n/3}}}{4n\sqrt{3}}\sum_{t=0}^{\infty}\frac{g(t)}{\sqrt{n}^t}.$$
 This is done in Section \ref{sect1} by computing separately $g(2t)$ and $g(2t+1)$. In spite of having a double sum representation for $g(t)$, we will see that the coefficients $g(t)$ are indeed equal to $\omega_t$ as in Theorem \ref{Sullivanthm0}. 
 
The next step is to estimate the number $g(t)$ in the following form:
\begin{equation}\label{fjineq}
f(t)-l(t)\leq g(t)\leq f(t)+u(t).
\end{equation}
Here $f(t)$ has the property that $\lim_{t\to \infty}\frac{g(t)}{f(t)}=1$, $\lim_{t\to \infty}\frac{l(t)}{f(t)}=0$, and $\lim_{t\to \infty}\frac{u(t)}{f(t)}=0$.
Precise descriptions for $f(t)$, $u(t)$, and $l(t)$ are given in Section \ref{sect3} along with the inequalities of the form \eqref{fjineq}. In order to prove such inequalities, we will use the preliminary lemmas from Section \ref{sect2} and the summation package \texttt{Sigma}. 

Finally in Section \ref{sec7}, applying the bounds for $g(t)$, given in Section \ref{sect3}, we find $\widehat{L}_1(k), \widehat{U}_1(k)$ such that
$$\frac{\widehat{L}_1(k)}{\sqrt{n}^k}<\sum_{t=k}^{\infty}\frac{g(t)}{\sqrt{n}^t}<\frac{\widehat{U}_1(k)}{\sqrt{n}^k}.$$
Also we compute explicitly $\widehat{L}_2(k)$ and $\widehat{U}_2(k)$ such that
$$\frac{e^{\pi\sqrt{2n/3}}}{4n\sqrt{3}}\frac{\widehat{L}_2(k)}{\sqrt{n}^k}<\dfrac{\sqrt{12}\ e^{\mu(n)}}{24n-1}\frac{1}{\mu(n)^k}<\frac{e^{\pi\sqrt{2n/3}}}{4n\sqrt{3}}\frac{\widehat{U}_2(k)}{\sqrt{n}^k}.$$
Combining the error bounds as $L(k)=\widehat{L}_1(k)+\widehat{L}_2(k)$ and $U(k)=\widehat{U}_1(k)+\widehat{U}_2(k)$, we arrive at the desired inequality \eqref{goaleqn1} for $p(n)$.
 \section{Estimation of the coefficients $g(t)$}\label{sect1}
From Theorem \ref{bprz1}, we have for all $k\in\mathbb{Z}_{\geq 2}$ and  $n> \widehat{g}(k)$ such that $(n,k)\neq (6,2)$,
\begin{equation}\label{eqn1}
\dfrac{\sqrt{12}\ e^{\mu(n)}}{24n-1}\Bigl(1-\dfrac{1}{\mu(n)}-\dfrac{1}{\mu(n)^k}\Bigr)<p(n)<\dfrac{\sqrt{12}\ e^{\mu(n)}}{24n-1}\Bigl(1-\dfrac{1}{\mu(n)}+\dfrac{1}{\mu(n)^k}\Bigr).
\end{equation}
Rewrite the major term $\dfrac{\sqrt{12}\ e^{\mu(n)}}{24n-1}\Bigl(1-\dfrac{1}{\mu(n)}\Bigr)$ in the following way: 
\begin{equation}\label{eqn2}
\dfrac{\sqrt{12}\ e^{\mu(n)}}{24n-1}\Bigl(1-\dfrac{1}{\mu(n)}\Bigr)=\dfrac{1}{4n\sqrt{3}}e^{\pi\sqrt{2n/3}}\ \underset{:=A_1(n)}{\underbrace{e^{\pi\sqrt{2n/3}\ \bigl(\sqrt{1-\frac{1}{24n}}-1\bigr)}}} \underset{:=A_2(n)}{\underbrace{\Bigl(1-\dfrac{1}{24n}\Bigr)^{-1}\Bigl(1-\dfrac{1}{\mu(n)}\Bigr)}}.
\end{equation}
Next we compute the Taylor expansion of the residue parts of $A_1(n)$ and $A_2(n)$, defined in \eqref{eqn2}.\\
  \begin{definition}\label{newdef1}
For $t\in \mathbb{Z}_{\geq 0}$, define
\begin{equation}\label{defeqn0}
e_1(t) :=
\begin{cases}
1, &\quad \text{if}\ t=0\\
\dfrac{(-1)^t}{(24)^t}\dfrac{(1/2-t)_{t+1}}{t}\displaystyle \sum_{u=1}^{t}\dfrac{(-1)^u (-t)_u}{(t+u)! (2u-1)!}\Bigl(\dfrac{\pi^2}{36}\Bigr)^u, &\quad \text{otherwise}
\end{cases},
\end{equation}
and
\begin{equation}\label{defeqn1}
E_1\Bigl(\dfrac{1}{\sqrt{n}}\Bigr):= \sum_{t=0}^{\infty}e_1(t) \Bigl(\dfrac{1}{\sqrt{n}}\Bigr)^{2t}, \ n\geq 1.
\end{equation}	
\end{definition}
\begin{definition}\label{newdef2}
For $t\in \mathbb{Z}_{\geq 0}$, define
\begin{equation}\label{defeqn2}
o_1(t) := -\dfrac{\pi}{12\sqrt{6}} \Biggl(\dfrac{(-1)^t(1/2-t)_{t+1}}{(24)^t}\sum_{u=0}^{t}\dfrac{(-1)^u(-t)_u}{(t+u+1)!(2u)!}\Bigl(\dfrac{\pi^2}{36}\Bigr)^u\Biggr)
\end{equation}
and
\begin{equation}\label{defeqn3}
O_1\Bigl(\dfrac{1}{\sqrt{n}}\Bigr):= \sum_{t=0}^{\infty}o_1(t) \Bigl(\dfrac{1}{\sqrt{n}}\Bigr)^{2t+1}, \ n\geq 1.
\end{equation}	
\end{definition}
\begin{lemma}\label{PP1}
  For $j$, $k\in \mathbb{Z}_{\geq 0}$,
  \begin{equation}\label{PPP1}
    \sum_{i=0}^k(-1)^i\binom{k}{i}\binom{i/2}{j}=\left\{\begin{array}{cc} 1, & j=k=0\\ (-1)^j 2^{k-2j}\frac{k}{j}\binom{2j-k-1}{j-k}, & \text{otherwise}\end{array}\right..
    \end{equation}
\end{lemma}
\begin{proof}
  The case $j=k=0$ is trivial. By the inversion relation
  $$f(k)=\sum_{i=0}^k(-1)^i\binom{k}{i}g(i) \Leftrightarrow g(k)=\sum_{i=0}^k (-1)^i\binom{k}{i}f(i),$$
  \eqref{PPP1} for $j\neq 0$ is equivalent to
  $$\sum_{i=0}^k(-1)^{i+j}2^{i-2j}\frac{i}{j}\binom{k}{i}\binom{2j-i-1}{j-i}=\binom{k/2}{j};$$
  which can be proved (and derived) by any standard summation method, resp. algorithm.
  \end{proof}
\begin{lemma}\label{lemsec1}
Let $A_1(n)$ be defined as in \eqref{eqn2}. Let $E_1(n)$ be as in Definition \ref{newdef1} and $O_1(n)$ as in Definition \ref{newdef2}. Then
\begin{equation}\label{eqn3}
A_1(n)= E_1\Bigl(\dfrac{1}{\sqrt{n}}\Bigr)+O_1\Bigl(\dfrac{1}{\sqrt{n}}\Bigr).
\end{equation}
\end{lemma}
\begin{proof}
From Equation \eqref{eqn2}, we get
\begin{eqnarray}\label{eqn4}
A_1(n) &=&e^{\pi\sqrt{2n/3}\ \bigl(\sqrt{1-\frac{1}{24n}}-1\bigr)}\nonumber\\
&=&\sum_{k=0}^{\infty}\dfrac{(\pi\sqrt{2n/3})^k}{k!}\Biggl(\sqrt{1-\frac{1}{24n}}-1\Biggr)^k\nonumber\\
&=& \sum_{k=0}^{\infty}\dfrac{(\pi\sqrt{2/3})^k}{k!}(\sqrt{n})^k\sum_{i=0}^{k}\binom{k}{i}(-1)^{k-i}\Biggl(\sqrt{1-\frac{1}{24n}}\Biggr)^i\nonumber\\
&=& \sum_{k=0}^{\infty}\dfrac{(\pi\sqrt{2/3})^k}{k!}(\sqrt{n})^k\sum_{i=0}^{k}\binom{k}{i}(-1)^{k-i} \sum_{j=0}^{\infty}\binom{i/2}{j}\dfrac{(-1)^j}{(24n)^j}\nonumber\\
&=& \sum_{k=0}^{\infty}\sum_{i=0}^{k}\sum_{j=0}^{\infty}\dfrac{(\pi\sqrt{2/3})^k}{k!} \dfrac{(-1)^{k-i+j}}{(24)^j}\binom{k}{i}\binom{i/2}{j}(\sqrt{n})^{k-2j}.
\end{eqnarray}
Define $S:=\bigl\{(k,i,j)\in \mathbb{Z}^3_{\geq 0} : 0 \leq i \leq k\bigr \}$. In order to express $A_1(n)$ in the form $\sum_{m=0}^{\infty}a_m(\frac{1}{\sqrt{n}})^m$, we split the set $S$ into a disjoint union of subsets; i.e., $S:=\underset{t \in \mathbb{Z}_{\geq 0}}{\bigcup}V(t)$, where for each $t\in\mathbb{Z}_{\geq 0}$,
$V(t):=\bigl\{(k,i,j)\in \mathbb{Z}^3_{\geq 0} : k-2j =-t \bigr\}.$

Notice that for $k>j$, by Lemma \ref{PP1}, $\sum_{i=0}^k\binom{k}{i}\binom{i/2}{j}=0$. Furthermore, for each element $r =(k,i,j)\in S$, we define 
\begin{equation*}
S(r):= \dfrac{(\pi\sqrt{2/3})^k}{k!} \dfrac{(-1)^{k-i+j}}{(24)^j}\binom{k}{i}\binom{i/2}{j}\ \ \text{and}\ \ f(r):=k-2j.
\end{equation*}
Rewrite \eqref{eqn4} as
\begin{eqnarray}\label{eqn5}
A_1(n) &=& \sum_{r \in S} S(r) (\sqrt{n})^{f(r)}= \sum_{t=0}^{\infty}\sum_{r\in V(t)} S(r) \Bigl(\dfrac{1}{\sqrt{n}}\Bigr)^t\nonumber\\
&=& \sum_{t=0}^{\infty}\underset{r\in V(2t)}{\sum} S(r) \Bigl(\dfrac{1}{\sqrt{n}}\Bigr)^{2t}+\sum_{t=0}^{\infty}\underset{r\in V(2t+1)}{\sum} S(r) \Bigl(\dfrac{1}{\sqrt{n}}\Bigr)^{2t+1}.
\end{eqnarray}
Now
\begin{eqnarray}\label{eqn6}
V(2t) &=& \bigl\{(k,i,j)\in S : k-2j=-2t\bigr\}\nonumber\\
&=& \bigl\{(k,i,j)\in S :k \equiv 0\ (\text{mod}\ 2)\ \text{and}\  k-2j=-2t\bigr\}\nonumber\\
&=& \bigl\{(2u,i,j)\in S :  j=u+t\bigr\}= \bigl\{(2u,i,u+t)\in \mathbb{Z}^3_{\geq 0} : 0 \leq i \leq 2u\bigr\}.
\end{eqnarray}
From \eqref{eqn6}, it follows that
\begin{eqnarray}\label{eqn7}
& &\sum_{t=0}^{\infty}\underset{r\in V(2t)}{\sum} S(r) \Bigl(\dfrac{1}{\sqrt{n}}\Bigr)^{2t}\nonumber\\
&=&  \sum_{t=0}^{\infty} \dfrac{(-1)^t}{(24)^t}\Biggl(\sum_{u=0}^{\infty}\dfrac{(2\pi^2/3)^u}{(2u)!}\dfrac{(-1)^u}{(24)^u}\sum_{i=0}^{2u}(-1)^i \binom{2u}{i}\binom{i/2}{u+t}\Biggr)\Bigl(\dfrac{1}{\sqrt{n}}\Bigr)^{2t}\nonumber\\
&=&  \sum_{t=0}^{\infty} \dfrac{(-1)^t}{(24)^t}\Biggl(\sum_{u=0}^{\infty}\dfrac{(-1)^u}{(2u)!}\Bigl(\dfrac{\pi}{6}\Bigr)^{2u}\underset{:=\mathcal{E}_1(u,t)}{\underbrace{\sum_{i=0}^{2u}(-1)^i \binom{2u}{i}\binom{i/2}{u+t}}}\Biggr)\Bigl(\dfrac{1}{\sqrt{n}}\Bigr)^{2t}.\nonumber\\
\end{eqnarray}
By Lemma \ref{PP1},
$$\mathcal{E}_1(u,t)=\left\{\begin{array}{cc} 1, & \text{if $u=t=0$}\\ 0, & \text{if $u>t$}\\ \frac{2u(1/2-t)_{t+1}(-t)_u}{t(t+u)!}, &  \text{otherwise}\end{array}\right..$$
Consequently, for all $t\geq 1$, 
\begin{equation}\label{eqn10}
\sum_{u=0}^{t}\dfrac{(-1)^u}{(2u)!}\Bigl(\dfrac{\pi}{6}\Bigr)^{2u}\mathcal{E}_1(u,t) = \dfrac{(1/2-t)_{t+1}}{t}\sum_{u=1}^{t}\dfrac{(-1)^u (-t)_u}{(t+u)! (2u-1)!}\Bigl(\dfrac{\pi^2}{36}\Bigr)^u.
\end{equation}
It follows that
\begin{equation}\label{eqn11}
\begin{split}
\sum_{t=0}^{\infty}\underset{r\in V(2t)}{\sum} S(r) \Bigl(\dfrac{1}{\sqrt{n}}\Bigr)^{2t} &=1+ \sum_{t=1}^{\infty} \Biggl(\dfrac{(-1)^t}{(24)^t}\dfrac{(1/2-t)_{t+1}}{t}\sum_{u=1}^{t}\dfrac{(-1)^u (-t)_u}{(t+u)! (2u-1)!}\Bigl(\dfrac{\pi^2}{36}\Bigr)^u\Biggr)\Bigl(\dfrac{1}{\sqrt{n}}\Bigr)^{2t}\\
&= E_1\Bigl(\dfrac{1}{\sqrt{n}}\Bigr).
\end{split}
\end{equation}
Similar to \eqref{eqn6}, we have
\begin{equation}\label{eqn12}
V(2t+1) = \bigl\{(2u+1,i,u+t+1)\in \mathbb{Z}^3_{\geq 0}: 0\leq i \leq 2u+1\bigr \},
\end{equation}
and consequently, it follows that
\begin{eqnarray}\label{eqn13}
& &\sum_{t=0}^{\infty}\underset{r\in V(2t+1)}{\sum} S(r) \Bigl(\dfrac{1}{\sqrt{n}}\Bigr)^{2t+1}\nonumber\\
&=& \sum_{t=0}^{\infty} \dfrac{(-1)^t}{(24)^t}\Biggl(\sum_{u=0}^{\infty}\dfrac{(\pi \sqrt{2/3})^{2u+1}}{(2u+1)!}\dfrac{(-1)^u}{(24)^{u+1}}\underset{:=\mathcal{O}_1(u,t)}{\underbrace{\sum_{i=0}^{2u+1}(-1)^i \binom{2u+1}{i}\binom{i/2}{u+t+1}}}\Biggr)\Bigl(\dfrac{1}{\sqrt{n}}\Bigr)^{2t+1}.\nonumber\\
\end{eqnarray}
By Lemma \ref{PP1},
$$\mathcal{O}_1(u,t)=\left\{\begin{array}{cc} 0, & \text{if $u>t$}\\ -\frac{(2u+1)(1/2-t)_{t+1}(-t)_u}{(t+u+1)!}, & \text{otherwise}\end{array}\right..$$
It follows that
\begin{eqnarray}\label{eqn16}
& &\sum_{t=0}^{\infty}\underset{r\in V(2t+1)}{\sum} S(r) \Bigl(\dfrac{1}{\sqrt{n}}\Bigr)^{2t+1}\nonumber\\
& &=-\dfrac{\pi}{12\sqrt{6}} \sum_{t=0}^{\infty} \Biggl(\dfrac{(-1)^t(1/2-t)_{t+1}}{(24)^t}\sum_{u=0}^{t}\dfrac{(-1)^u(-t)_u}{(t+u+1)!(2u)!}\Bigl(\dfrac{\pi^2}{36}\Bigr)^u\Biggr)\Bigl(\dfrac{1}{\sqrt{n}}\Bigr)^{2t+1}\nonumber\\
& &= O_1\Bigl(\dfrac{1}{\sqrt{n}}\Bigr).
\end{eqnarray}
From \eqref{eqn5}, \eqref{eqn11}, and \eqref{eqn16}, we get \eqref{eqn3}.
\end{proof}
\begin{definition}\label{newdef3}
For $t \in \mathbb{Z}_{\geq 0}$, define
\begin{equation}\label{eqn18}
E_2\Bigl(\dfrac{1}{\sqrt{n}}\Bigr):=\sum_{t=0}^{\infty}e_2(t) \Bigl(\dfrac{1}{\sqrt{n}}\Bigr)^{2t}\ \text{with}\ \  e_2(t):=\dfrac{1}{(24)^t}.
\end{equation}
\end{definition}
\begin{definition}\label{newdef4}
	For $t \in \mathbb{Z}_{\geq 0}$, define
\begin{equation}\label{eqn19}
O_2\Bigl(\dfrac{1}{\sqrt{n}}\Bigr):=\sum_{t=0}^{\infty}o_2(t) \Bigl(\dfrac{1}{\sqrt{n}}\Bigr)^{2t+1}\ \text{with}\ \  o_2(t):=-\dfrac{6}{\pi\sqrt{24}}\binom{-3/2}{t}\dfrac{(-1)^t}{(24)^t}.
\end{equation}
\end{definition}
\begin{lemma}\label{lemsec2}
Let $A_2(n)$ be defined as in \eqref{eqn2}. Let $E_2(n)$ be as in Definition \ref{newdef3} and $O_2(n)$ as in Definition \ref{newdef4}. Then	
	\begin{equation}\label{eqn17}
	A_2(n)= E_2\Bigl(\dfrac{1}{\sqrt{n}}\Bigr)+O_2\Bigl(\dfrac{1}{\sqrt{n}}\Bigr).
	\end{equation}
\end{lemma}
\begin{proof}
Recall the definition of $A_2(n)$ from \eqref{eqn2} and expand it in the following way: 
\begin{eqnarray}\label{eqn20}
A_2(n) &=& \Bigl(1-\dfrac{1}{24n}\Bigr)^{-1}\Bigl(1-\dfrac{1}{\mu(n)}\Bigr)= \Bigl(1-\dfrac{1}{24n}\Bigr)^{-1}-\dfrac{6}{\pi\sqrt{24}}\dfrac{1}{\sqrt{n}}\Bigl(1-\dfrac{1}{24n}\Bigr)^{-3/2}\nonumber\\
&=& \sum_{t=0}^{\infty}\dfrac{1}{(24)^t}\Bigl(\dfrac{1}{\sqrt{n}}\Bigr)^{2t}-\dfrac{6}{\pi\sqrt{24}}\sum_{t=0}^{\infty}\binom{-3/2}{t}\dfrac{(-1)^t}{(24)^t} \Bigl(\dfrac{1}{\sqrt{n}}\Bigr)^{2t+1} \nonumber\\
 &=& E_2\Bigl(\dfrac{1}{\sqrt{n}}\Bigr)+O_2\Bigl(\dfrac{1}{\sqrt{n}}\Bigr).
\end{eqnarray}
This completes the proof of \eqref{eqn17}.
\end{proof}
\begin{definition}\label{newdef5}
In view of the Definitions \ref{newdef1}-\ref{newdef4}, we define
\begin{equation}\label{eqn22}
S_{e,1}\Bigl(\dfrac{1}{\sqrt{n}}\Bigr):=E_1\Bigl(\dfrac{1}{\sqrt{n}}\Bigr)E_2\Bigl(\dfrac{1}{\sqrt{n}}\Bigr),
\end{equation}
\begin{equation}\label{eqn23}
S_{e,2}\Bigl(\dfrac{1}{\sqrt{n}}\Bigr):=O_1\Bigl(\dfrac{1}{\sqrt{n}}\Bigr)O_2\Bigl(\dfrac{1}{\sqrt{n}}\Bigr),
\end{equation}
\begin{equation}\label{eqn24}
S_{o,1}\Bigl(\dfrac{1}{\sqrt{n}}\Bigr):=E_1\Bigl(\dfrac{1}{\sqrt{n}}\Bigr)O_2\Bigl(\dfrac{1}{\sqrt{n}}\Bigr),
\end{equation}
and
\begin{equation}\label{eqn25}
S_{o,2}\Bigl(\dfrac{1}{\sqrt{n}}\Bigr):=E_2\Bigl(\dfrac{1}{\sqrt{n}}\Bigr)O_1\Bigl(\dfrac{1}{\sqrt{n}}\Bigr).
\end{equation}	
\end{definition}

\begin{lemma}\label{newlemma1}
For each $i \in \{1,2\}$, let $S_{e,i}\Bigl(\dfrac{1}{\sqrt{n}}\Bigr)$ and $ S_{o,i}\Bigl(\dfrac{1}{\sqrt{n}}\Bigr)$ be as in Definition \ref{newdef5}. Then
\begin{equation}\label{neweqn1}
\dfrac{\sqrt{12}\ e^{\mu(n)}}{24n-1}\Bigl(1-\dfrac{1}{\mu(n)}\Bigr)=\dfrac{1}{4n\sqrt{3}}e^{\pi\sqrt{2n/3}} \sum_{i=1}^{2}\Biggl(S_{e,i}\Bigl(\dfrac{1}{\sqrt{n}}\Bigr)+S_{o,i}\Bigl(\dfrac{1}{\sqrt{n}}\Bigr)\Biggr).
\end{equation}
\end{lemma}
\begin{proof}
The proof follows immediately by applying Lemmas \ref{lemsec1} and \ref{lemsec2} to \eqref{eqn2}.
\end{proof}
\begin{definition}\label{newdef6}
		For $t \in \mathbb{Z}_{\geq 0}$, define
\begin{equation}\label{eqn28new}
S_1(t):=\sum_{s=1}^{t}\dfrac{(-1)^s (1/2-s)_{s+1}}{s}\sum_{u=1}^{s}\dfrac{(-1)^u (-s)_u}{(s+u)! (2u-1)!}\Bigl(\dfrac{\pi^2}{36}\Bigr)^u,
\end{equation}
and
\begin{equation}\label{eqn28}
\hspace{-5.5 cm} g_{e,1}(t) := \dfrac{1}{(24)^t}\Bigl(1+S_1(t)\Bigr).
\end{equation}	
\end{definition}
\begin{lemma}\label{lemsec3}
Let $S_{e,1}\Bigl(\dfrac{1}{\sqrt{n}}\Bigr)$ be as in \eqref{eqn22}. Let $g_{e,1}(t)$ be as in Definition \ref{newdef6}. Then
\begin{equation}\label{eqn26}
S_{e,1}\Bigl(\dfrac{1}{\sqrt{n}}\Bigr) =\sum_{t=0}^{\infty}g_{e,1}(t) \Bigl(\dfrac{1}{\sqrt{n}}\Bigr)^{2t}.
\end{equation}
\end{lemma}
\begin{proof}
From \eqref{defeqn1}, \eqref{eqn18}, and \eqref{eqn22}, we have
\begin{eqnarray}\label{eqn29}
S_{e,1}\Bigl(\dfrac{1}{\sqrt{n}}\Bigr) & = & E_1\Bigl(\dfrac{1}{\sqrt{n}}\Bigr)E_2\Bigl(\dfrac{1}{\sqrt{n}}\Bigr)= \Biggl(1+\sum_{t=1}^{\infty}e_1(t) \Bigl(\dfrac{1}{\sqrt{n}}\Bigr)^{2t}\Biggr) \Biggl(1+\sum_{t=1}^{\infty}e_2(t) \Bigl(\dfrac{1}{\sqrt{n}}\Bigr)^{2t}\Biggr)\nonumber\\
&=& 1+ \sum_{t=1}^{\infty} \Bigl(e_1(t)+e_2(t)\Bigr)\Bigl(\dfrac{1}{\sqrt{n}}\Bigr)^{2t}+ \sum_{t=2}^{\infty} \Biggl(\sum_{s=1}^{t-1}e_1(s)e_2(t-s)\Biggr)\Bigl(\dfrac{1}{\sqrt{n}}\Bigr)^{2t}\nonumber\\
&=& 1+ \sum_{t=1}^{\infty} \Biggl(e_1(t)+e_2(t)+\sum_{s=1}^{t-1}e_1(s)e_2(t-s)\Biggr)\Bigl(\dfrac{1}{\sqrt{n}}\Bigr)^{2t}.\nonumber\\
\end{eqnarray}
Combining \eqref{defeqn0} and \eqref{eqn18}, we obtain
\begin{eqnarray}\label{eqn30}
\hspace{-0.5 cm} e_1(t)+e_2(t)+\sum_{s=1}^{t-1}e_1(s)e_2(t-s)\hspace{-0.25 cm}&=&\hspace{-0.25 cm}\dfrac{(-1)^t(1/2-t)_{t+1}}{(24)^t\ t} \sum_{u=1}^{t}\dfrac{(-1)^u (-t)_u}{(t+u)! (2u-1)!}\Bigl(\dfrac{\pi^2}{36}\Bigr)^u+\dfrac{1}{(24)^t}\nonumber\\
&+&\hspace{-0.25 cm}\frac{1}{24^t}\sum_{s=1}^{t-1}\Biggl(\dfrac{(-1)^s (1/2-s)_{s+1}}{ s}\sum_{u=1}^{s}\dfrac{(-1)^u (-s)_u}{(s+u)! (2u-1)!}\Bigl(\dfrac{\pi^2}{36}\Bigr)^u\Biggr)\nonumber\\
  &=&\dfrac{1}{(24)^t}\Bigl(1+S_1(t)\Bigr)= g_{e,1}(t),
\end{eqnarray}
which concludes the proof of \eqref{eqn26}.
\end{proof}

\begin{definition}\label{newdef7}
	For $t \in \mathbb{Z}_{\geq 1}$, define
	\begin{equation}\label{eqn32new}
	S_2(t):=\sum_{s=0}^{t-1}(1/2-s)_{s+1}\binom{-3/2}{t-s-1}\sum_{u=0}^{s}\dfrac{(-1)^u(-s)_u}{(s+u+1)! (2u)!}\Bigl(\dfrac{\pi^2}{36}\Bigr)^{u},
	\end{equation}
	and
	\begin{equation}\label{eqn32}
\hspace{-7.5 cm}	g_{e,2}(t) := \dfrac{(-1)^{t-1}}{(24)^t}S_2(t).
	\end{equation}	
\end{definition}

\begin{lemma}\label{lemsec4}
Let $S_{e,2}\Bigl(\dfrac{1}{\sqrt{n}}\Bigr)$ as in \eqref{eqn23} and $g_{e,2}(t)$ as in Definition \ref{newdef7}. Then	
	\begin{equation}\label{eqn31}
	S_{e,2}\Bigl(\dfrac{1}{\sqrt{n}}\Bigr) =\sum_{t=1}^{\infty}g_{e,2}(t) \Bigl(\dfrac{1}{\sqrt{n}}\Bigr)^{2t}.
	\end{equation}
	\end{lemma}
\begin{proof}
From \eqref{defeqn3}, \eqref{eqn17} and \eqref{eqn23}, we have
\begin{eqnarray}\label{eqn33}
S_{e,2}\Bigl(\dfrac{1}{\sqrt{n}}\Bigr) & = & O_1\Bigl(\dfrac{1}{\sqrt{n}}\Bigr)O_2\Bigl(\dfrac{1}{\sqrt{n}}\Bigr)\nonumber\\
&=& \Biggl(\sum_{t=0}^{\infty}o_1(t) \Bigl(\dfrac{1}{\sqrt{n}}\Bigr)^{2t+1}\Biggr) \Biggl(\sum_{t=0}^{\infty}o_2(t) \Bigl(\dfrac{1}{\sqrt{n}}\Bigr)^{2t+1}\Biggr)\nonumber\\
&=& \sum_{t=1}^{\infty} \Biggl(\sum_{s=0}^{t-1}o_1(s)o_2(t-s-1)\Biggr)\Bigl(\dfrac{1}{\sqrt{n}}\Bigr)^{2t}\nonumber\\
&=& \sum_{t=1}^{\infty} g_{e,2}(t) \Bigl(\dfrac{1}{\sqrt{n}}\Bigr)^{2t}\ (\text{by}\ \ \eqref{defeqn2}\ \text{and}\ \eqref{eqn19}).
\end{eqnarray}
\end{proof}

\begin{definition}\label{newdef8}
	For $t \in \mathbb{Z}_{\geq 2}$, define
\begin{equation}\label{eqn36new}
S_3(t):=\sum_{s=1}^{t}\dfrac{ (1/2-s)_{s+1}\binom{-3/2}{t-s}}{s}\sum_{u=1}^{s}\dfrac{(-1)^u (-s)_u}{(s+u)! (2u-1)!}\Bigl(\dfrac{\pi^2}{36}\Bigr)^u,
\end{equation}
and
\begin{equation}\label{eqn36}
\hspace{-1.3 cm} g_{o,1}(t):=
\begin{cases}
-\dfrac{6}{\pi \sqrt{24}}\dfrac{(-1)^t}{(24)^t}\biggl(\binom{-3/2}{t}+S_3(t)\biggr), & \quad \hspace{-0.5 cm} \text{if\ $t \geq 2$}\\
-\dfrac{432+\pi^2}{2304\sqrt{6}\pi}, & \quad \hspace{-0.5 cm} \text{if\ $t = 1$}\\
-\dfrac{6}{\pi\sqrt{24}}, & \quad \hspace{-0.5 cm} \text{if\ $t = 0$}	
\end{cases}.
\end{equation}
\end{definition}
\begin{lemma}\label{lemsec5}
Let $S_{o,1}\Bigl(\dfrac{1}{\sqrt{n}}\Bigr)$ as in \eqref{eqn24} and $g_{o,1}(t)$ be as in Definition \ref{newdef8}. Then	
	\begin{equation}\label{eqn34}
	S_{o,1}\Bigl(\dfrac{1}{\sqrt{n}}\Bigr) =\sum_{t=0}^{\infty}g_{o,1}(t) \Bigl(\dfrac{1}{\sqrt{n}}\Bigr)^{2t+1}.
	\end{equation}
\end{lemma}
\begin{proof}
From \eqref{defeqn1}, \eqref{eqn19} and \eqref{eqn24}, it follows that
\begingroup
\allowdisplaybreaks
\begin{eqnarray}\label{eqn37}
S_{o,1}\Bigl(\dfrac{1}{\sqrt{n}}\Bigr) & = & E_1\Bigl(\dfrac{1}{\sqrt{n}}\Bigr)O_2\Bigl(\dfrac{1}{\sqrt{n}}\Bigr)\nonumber\\
&=& \dfrac{1}{\sqrt{n}}\Biggl(1+\sum_{t=1}^{\infty}e_1(t) \Bigl(\dfrac{1}{\sqrt{n}}\Bigr)^{2t}\Biggr) \Biggl(-\dfrac{6}{\pi\sqrt{24}}+\sum_{t=1}^{\infty}o_2(t) \Bigl(\dfrac{1}{\sqrt{n}}\Bigr)^{2t}\Biggr)\nonumber\\
&=& -\dfrac{6}{\pi\sqrt{24}}\dfrac{1}{\sqrt{n}}-\dfrac{432+\pi^2}{2304\sqrt{6}\pi}\dfrac{1}{\sqrt{n}^3}+\sum_{t=2}^{\infty} \Biggl(o_2(t)+\sum_{s=1}^{t}e_1(s)o_2(t-s)\Biggr)\Bigl(\dfrac{1}{\sqrt{n}}\Bigr)^{2t+1}\nonumber\\
&=& -\dfrac{6}{\pi\sqrt{24}}\dfrac{1}{\sqrt{n}}-\dfrac{432+\pi^2}{2304\sqrt{6}\pi}\dfrac{1}{\sqrt{n}^3}+\sum_{t=2}^{\infty} g_{o,1}(t) \Bigl(\dfrac{1}{\sqrt{n}}\Bigr)^{2t+1} (\text{by}\ \ \eqref{defeqn0}\ \text{and}\ \eqref{eqn19}).\nonumber\\
\end{eqnarray}
\endgroup 
\end{proof}
\begin{definition}\label{newdef9}
	For $t \in \mathbb{Z}_{\geq 1}$, define
	\begin{equation}\label{eqn40new}
	S_4(t):=\sum_{s=0}^{t}(-1)^s (1/2-s)_{s+1}\sum_{u=0}^{s}\dfrac{(-1)^u(-s)_u}{(s+u+1)!(2u)!}\Bigl(\dfrac{\pi^2}{36}\Bigr)^{u},
	\end{equation}
	and
	\begin{equation}\label{eqn40}
	\hspace{-5.5 cm} g_{o,2}(t) := -\dfrac{\pi}{12 \sqrt{6}}\dfrac{1}{(24)^t}S_4(t).
	\end{equation}
\end{definition}
\begin{lemma}\label{lemsec6}
Let $S_{o,2}\Bigl(\dfrac{1}{\sqrt{n}}\Bigr)$ be as in \eqref{eqn25} and $g_{o,2}(t)$ be as in Definition \ref{newdef9}. Then	
	\begin{equation}\label{eqn38}
	S_{o,2}\Bigl(\dfrac{1}{\sqrt{n}}\Bigr) =\sum_{t=0}^{\infty}g_{o,2}(t) \Bigl(\dfrac{1}{\sqrt{n}}\Bigr)^{2t+1}.
	\end{equation}
\end{lemma}
\begin{proof}
	From \eqref{defeqn3}, \eqref{eqn18} and \eqref{eqn25}, it follows that
	\begin{eqnarray}\label{eqn41}
	S_{o,1}\Bigl(\dfrac{1}{\sqrt{n}}\Bigr) & = & O_1\Bigl(\dfrac{1}{\sqrt{n}}\Bigr)E_2\Bigl(\dfrac{1}{\sqrt{n}}\Bigr)\nonumber\\
	&=& \sum_{t=0}^{\infty}\Biggl(\sum_{s=0}^{t}o_1(s)e_2(t-s)\Biggr) \Bigl(\dfrac{1}{\sqrt{n}}\Bigr)^{2t+1}\nonumber\nonumber\\
	&=& \sum_{t=0}^{\infty} g_{o,2}(t) \Bigl(\dfrac{1}{\sqrt{n}}\Bigr)^{2t+1}\ (\text{by}\ \ \eqref{defeqn3}\ \text{and}\ \eqref{eqn18}).
	\end{eqnarray}
\end{proof}
\begin{definition}\label{newdef10}
For each $i \in \{1,2\}$, let $g_{e,i}(t)$ and $g_{o,i}(t)$ be as in Definitions \ref{newdef6}-\ref{newdef9}.	
We define a power series 
$$G(n):=\sum_{t=0}^{\infty}g(t)\Bigl(\dfrac{1}{\sqrt{n}}\Bigr)^t=\sum_{t=0}^{\infty}g(2t)\Bigl(\dfrac{1}{\sqrt{n}}\Bigr)^{2t}+\sum_{t=0}^{\infty}g(2t+1)\Bigl(\dfrac{1}{\sqrt{n}}\Bigr)^{2t+1},$$ 
where
\begin{equation}\label{neweqn2}
g(2t):= g_{e,1}(t)+g_{e,2}(t)\ \ \text{and}\ \
g(2t+1):= g_{o,1}(t)+g_{o,2}(t).
\end{equation}
\end{definition}
\begin{lemma}\label{newlemma2}
Let $G(n)$ be as in Definition \ref{newdef10}. Then 
\begin{equation}\label{newlemma2eqn1}
\dfrac{\sqrt{12}\ e^{\mu(n)}}{24n-1}\Bigl(1-\dfrac{1}{\mu(n)}\Bigr)=\dfrac{1}{4n\sqrt{3}}e^{\pi\sqrt{2n/3}} \cdot G(n).
\end{equation}
\end{lemma}
\begin{proof}
Applying Lemmas \ref{lemsec3}-\ref{lemsec6} to Lemma \ref{lemsec2}, we immediately obtain \eqref{newlemma2eqn1}.
\end{proof}
\begin{remark}\label{Remarkp(n)asymp}
Note that using \textnormal{\texttt{Sigma}} and \textnormal{\texttt{GeneratingFunctions}} due to \textnormal{Mallinger} \cite{Mallinger1996}, we observe that for all $t \geq 0$,
\begin{equation}\label{rmrkpnasympeqn1}
g(2t)=g_{e,1}(t)+g_{e,2}(t)=\omega_{2t}\ \ \text{and}\ \ g(2t+1)=g_{o,1}(t)+g_{o,2}(t)=\omega_{2t+1},
\end{equation} 
where $\omega_t$ is as in \eqref{Sullivan2}. Equivalently,
\begin{equation}\label{g(t)definition}
g(t)=\omega_t=\dfrac{1}{(-4\sqrt{6})^t}\sum_{k=0}^{\frac{t+1}{2}}\binom{t+1}{k}\dfrac{t+1-k}{(t+1-2k)!}\Bigl(\frac{\pi}{6}\Bigr)^{t-2k}.
\end{equation} 
However this was already clear from the uniqueness of the asymptotic expansion for $p(n)$ and its proof can be considered as an additional verification of our computations. The reader might wonder at this point why we did not use the single sum expression found by O'Sullivan to bound the remainder of the asymptotic expansion for $p(n)$. We tried this indeed, but could not obtain from $\omega_t$ an effective upper and lower bound. The summation package \textnormal{\texttt{Sigma}} could not rewrite $\omega_t$ as a definite sum which is crucial for our estimations. However going to the double sum expression $g(t)$, \textnormal{\texttt{Sigma}} was able to give a definite sum expression for the inner sum as we will see later, and this enabled us to obtain effective upper and lower bounds in the sense that we described earlier. Namely, $l(t)<g(t)<u(t)$ and $\lim_{t\to \infty}\frac{l(t)}{g(t)}=\lim_{t\to \infty}\frac{u(t)}{g(t)}=1$.   
\end{remark}
\section{Preliminary Lemmas}\label{sect2}
This section presents all the preliminary facts needed for the proofs of the lemmas stated in Section \ref{sect3}. The proofs of Lemmas \ref{lem1} to \ref{helperC}, except \ref{helperC2}, are presented in Subsection \ref{proofoflemmas}.
\begin{lemma}\label{lem1}
  Let $x_1,x_2,\dots,x_n\leq 1$ and $y_1,\dots,y_1$ be non-negative real numbers. Then
  $$\frac{(1-x_1)(1-x_2)\cdots (1-x_n)}{(1+y_1)(1+y_2)\cdots(1+y_n)}\geq 1-\sum_{j=1}^nx_j-\sum_{j=1}^ny_j.$$
\end{lemma}
\begin{lemma}\label{lem2}
  For $t\geq 1$ and non-negative integer $u \leq t$, we have
  $$\frac{1}{2t}\geq \frac{t(-t)_u(-1)^u}{(1+2t)(t+u)(t)_u}\geq \frac{1}{2t}\Biggl(1-\frac{u^2+\frac{1}{2}}{t}\Biggr).$$
\end{lemma}
\begin{lemma}\label{lem3}
  For $t\geq 1$ and non-negative integer $u \leq t$, we have
  $$\frac{2u+1}{2t}\geq \frac{1}{1+2t}+\frac{2t}{1+2t}\sum_{i=1}^u\frac{(-t)_i(-1)^i}{(t+i)(t)_i}\geq \frac{2u+1}{2t}-\frac{4u^3+6u^2+8u+3}{12t^2}.$$
\end{lemma}
Throughout the rest of this paper,  
$$\alpha:=\frac{\pi}{6}.$$
\begin{lemma}\label{helperC2}
  We have
  \begin{equation*}
    \begin{split}
      &\sum_{u=0}^{\infty}\frac{\alpha^{2u}}{(2u)!}=\cosh(\alpha),\ \sum_{u=0}^{\infty}\frac{u\alpha^{2u}}{(2u)!}=\frac{1}{2}\alpha\sinh(\alpha), \ \sum_{u=0}^{\infty}\frac{u^2\alpha^{2u}}{(2u)!}=\frac{\alpha^2}{4}\cosh(\alpha)+\frac{\alpha}{4}\sinh(\alpha),\\
\intertext{and}
      &\sum_{u=0}^{\infty}\frac{u^3\alpha^{2u}}{(2u)!}=\frac{3\alpha^2}{8}\cosh(\alpha)+\frac{\alpha(\alpha^2+1)}{8}\sinh(\alpha).
    \end{split}
  \end{equation*}
\end{lemma}
\begin{lemma}\label{mainlemma2}
  Let $u\in\mathbb{Z}_{\geq 0}$. Assume that $a_{n+1}-a_n\geq b_{n+1}-b_n$ for all $n\geq u$, and $\lim_{n\to \infty}a_n=\lim_{n\to\infty}b_n=0$. Then
  $$b_n\geq a_n \text{ for all $n\geq u$}.$$
\end{lemma}
\begin{lemma}\label{helperC}
  For $t\geq 1$ and $k\in\{0,1,2,3\}$ we have
  $$\sum_{u=t+1}^{\infty}\frac{u^k\alpha^{2u}}{(2u)!}\leq \frac{C_k}{t^2}\ \ \text{with}\ \ C_k=\frac{\alpha^4 2^k}{18}.$$
\end{lemma}

\begin{lemma}\cite[Equation 7.5, Lemma 7.3]{BPRZ}\label{BPRZ}
	For $n,k,s \in \mathbb{Z}_{\geq 1}$ and $n >2s$ let
	$$b_{k,n}(s):=\dfrac{4\sqrt{s}}{\sqrt{s+k-1}}\binom{s+k-1}{s-1}\dfrac{1}{n^k},$$
	then 
	\begin{equation}\label{BPRZeqn1}
	0< \sum_{t=k}^{\infty} \binom{-\frac{2s-1}{2}}{t}\dfrac{(-1)^k}{n^k}<b_{k,n}(s).
	\end{equation}
\end{lemma}

\begin{lemma}\cite[Equation 7.9, Lemma 7.5]{BPRZ}\label{BPRZ1}
	For $m,n,s \in \mathbb{Z}_{\geq 1}$ and $n >2s$ let
	$$c_{m,n}(s):=\dfrac{2}{m}\dfrac{s^m}{n^m},$$
	then 
	\begin{equation}\label{BPRZ1eqn1}
	-\dfrac{c_{m,n}(s)}{\sqrt{m}}< \sum_{k=m}^{\infty} \binom{1/2}{k}\dfrac{(-1)^ks^k}{n^k}<0.
	\end{equation}
\end{lemma}

\begin{lemma}\cite[Equation 7.7, Lemma 7.4]{BPRZ}\label{BPRZ2}
	For $n,s \in \mathbb{Z}_{\geq 1}$, $m \in \mathbb{N}$ and $n >2s$ let
	$$\beta_{m,n}(s):=\dfrac{2}{n^m}\binom{s+m-1}{s-1},$$
	then 
	\begin{equation}\label{BPRZ2eqn1}
	0< \sum_{k=m}^{\infty} \binom{-s}{k}\dfrac{(-1)^k}{n^k}<\beta_{m,n}(s).
	\end{equation}
\end{lemma}

\section{Estimation of $\bigl(S_i(t)\bigr)$}\label{sect3}
For the sake of a compact representation the organization of this section is as follows. We first present the statements of the lemmas needed; then, in a separate subsection we present the proofs.

\subsection{The Lemmas \ref{S1lemma} to \ref{S4lemma}}

\begin{lemma}\label{S1lemma}
Let $S_1(t)$ be as in Definition \ref{newdef6}. Then for all $t\geq1$,
\begin{equation}\label{S1lemmaeqn}
-\frac{1}{8t^2}< \frac{S_1(t)}{(-1)^t\binom{-\frac{3}{2}}{t}}-\frac{(-1)^t}{\binom{-\frac{3}{2}}{t}}\bigl(\cosh(\alpha)-1\bigr)+\frac{1}{2t}\alpha\sinh(\alpha)< \frac{13}{25t^2}.
\end{equation}
\end{lemma}

\begin{lemma}\label{S2lemma}
	Let $S_2(t)$ be as in Definition \ref{newdef7}. Then for all $t \geq 1$,
	\begin{equation}\label{S2lemmaeqn}
	-\frac{11}{10t}< \frac{S_2(t)}{\binom{-\frac{3}{2}}{t}}-\frac{(-1)^{t}}{\binom{-\frac{3}{2}}{t}}\cosh(\alpha)+\frac{\sinh(\alpha)}{\alpha}<\frac{1}{t}.
	\end{equation}
\end{lemma}

\begin{lemma}\label{S3lemma}
Let $S_3(t)$ be as in Definition \ref{newdef8}. Then for all $t\geq 2$,
\begin{equation}\label{S3lemmaeqn}
-\frac{71}{100t}< \frac{S_3(t)}{\binom{-\frac{3}{2}}{t}}+\frac{(-1)^{t}}{\binom{-\frac{3}{2}}{t}}\alpha\sinh(\alpha)+1-\cosh(\alpha)< \frac{12}{25t}.
\end{equation}	
\end{lemma}

\begin{lemma}\label{S4lemma}
Let $S_4(t)$ be as in Definition \ref{newdef9}. Then for $t \geq 1$,
\begin{equation}\label{S4lemmaeqn}
-\frac{1}{3t^2}< \frac{S_4(t)}{(-1)^t\binom{-\frac{3}{2}}{t}}-\frac{(-1)^t}{\binom{-\frac{3}{2}}{t}}\frac{\sinh(\alpha)}{\alpha}+\frac{1}{2t}\cosh(\alpha)< \frac{13}{20t^2}.
\end{equation}
\end{lemma}

\subsection{The Proofs of Lemma \ref{S1lemma} to \ref{S4lemma}}

\emph{Proof of Lemma \ref{S1lemma}:}
 We rewrite $S_1(t)$ as follows:
 \begin{eqnarray}\label{S1lemprfeqn1}
S_1(t)&=&\sum_{u=1}^t\frac{(-1)^u\alpha^{2u}}{(2u-1)!} \sum_{s=u}^t\frac{(-1)^s}{s} \Bigl(\frac{1}{2} - s\Bigr)_{s + 1}\frac{(-s)_u}{(s + u)!}\nonumber\\
&=& \sum_{u=1}^t\frac{(-1)^u\alpha^{2u}}{(2u-1)!}\underbrace{\sum_{s=0}^{t-u}\frac{(-1)^{s+u}}{s+u} \Bigl(\frac{1}{2} - s-u\Bigr)_{s+u + 1}\frac{(-s-u)_u}{(s + 2u)!}\nonumber}_{=:S_1(t,u)}.\\
 \end{eqnarray}
We use the summation package \texttt{Sigma} (and its mechanization by \texttt{EvaluateMultiSums})\footnote[2]{For further explanations of this rigorous computer derviation we refer to Appendix \ref{reductionbySigma} and Remark \ref{remarkappendix}.}, to derive and prove that
\begin{equation}\label{S1lemprfeqn2}
 S_1(t,u)=(-1)^t\binom{-\frac{3}{2}}{t}\frac{(-1)^u}{2u}A_1(t,u),
\end{equation}
where
\begin{equation*}
A_1(t,u)=\frac{t(-t)_u(-1)^u}{(1+2t)(t+u)(t)_u}-\Biggl(\frac{(-1)^{t+1}}{\binom{-\frac{3}{2}}{t}}+\frac{1}{(1+2t)}+\frac{2t}{1+2t}\sum_{i=1}^u\frac{(-t)_i(-1)^i}{(t+i)(t)_i}\Biggr).
\end{equation*}
Now by Lemmas \ref{lem2} and \ref{lem3}, 
\begin{equation*}
\frac{1}{2t}+\frac{(-1)^t}{\binom{-\frac{3}{2}}{t}}-\frac{2u+1}{2t}-\frac{u^2+\frac{1}{2}}{2t^2} \leq A_1(t,u) \leq \frac{1}{2t}+\frac{(-1)^t}{\binom{-\frac{3}{2}}{t}}-\frac{2u+1}{2t}+\frac{4u^3+6u^2+8u+3}{12t^2}.
\end{equation*}
It is convenient to reorder the terms in this inequality with respect to the powers of $u$:
\begin{equation}\label{ss11d}
 \frac{(-1)^t}{\binom{-\frac{3}{2}}{t}}-\frac{1}{4t^2}-\frac{u}{t}-\frac{u^2}{2t^2} \leq A_1(t,u)\leq \frac{(-1)^t}{\binom{-\frac{3}{2}}{t}}+\frac{1}{4t^2}+u\Bigl(\frac{2}{3t^2}-\frac{1}{t}\Bigr)+\frac{u^2}{2t^2}+\frac{u^3}{3t^2}.
  \end{equation}
Combining \eqref{S1lemprfeqn1} and \eqref{S1lemprfeqn2}, if follows that
\begin{equation}\label{S1lemprfeqn3}
S_1(t)=(-1)^t\binom{-\frac{3}{2}}{t}\sum_{u=1}^t\frac{\alpha^{2u}A_1(t,u)}{(2u)!}.
\end{equation}
To derive a lower bound, combine \eqref{ss11d} with \eqref{S1lemprfeqn3} to get
\begingroup
\allowdisplaybreaks
\begin{eqnarray}\label{S1lemprfeqn4}
\dfrac{S_1(t)}{(-1)^t\binom{-\frac{3}{2}}{t}}\hspace{-0.25 cm}&\geq&\hspace{-0.25 cm} \Biggl(\frac{(-1)^t}{\binom{-\frac{3}{2}}{t}}-\frac{1}{4t^2}\Biggr)\sum_{u=1}^t\frac{\alpha^{2u}}{(2u)!}-\frac{1}{t}\sum_{u=1}^t\frac{u \alpha^{2u}}{(2u)!}-\frac{1}{2t^2}\sum_{u=1}^t\frac{u^2 \alpha^{2u}}{(2u)!}\nonumber\\
\hspace{-0.25 cm}&\geq&\hspace{-0.25 cm} \Biggl(\frac{(-1)^t}{\binom{-\frac{3}{2}}{t}}-\frac{1}{4t^2}\Biggr)\Biggl(\sum_{u=0}^{\infty}\frac{\alpha^{2u}}{(2u)!}-1-\sum_{u=t+1}^{\infty}\frac{\alpha^{2u}}{(2u)!}\Biggr)-\frac{1}{t}\sum_{u=0}^{\infty}\frac{u \alpha^{2u}}{(2u)!}-\frac{1}{2t^2}\sum_{u=0}^{\infty}\frac{u^2 \alpha^{2u}}{(2u)!}.\nonumber\\
\hspace{-0.25 cm}&>&\hspace{-0.25 cm} \Biggl(\frac{(-1)^t}{\binom{-\frac{3}{2}}{t}}-\frac{1}{4t^2}\Biggr)\Biggl(\sum_{u=0}^{\infty}\frac{\alpha^{2u}}{(2u)!}-1-\frac{\alpha^4}{18t^2}\Biggr)-\frac{1}{t}\sum_{u=0}^{\infty}\frac{u \alpha^{2u}}{(2u)!}-\frac{1}{2t^2}\sum_{u=0}^{\infty}\frac{u^2 \alpha^{2u}}{(2u)!}\nonumber\\
\hspace{-0.25 cm}&&\hspace{-0.25 cm} \hspace{4.5 cm} \Biggl(\text{by Lemma\ \ref{helperC}\ \  and}\ \ \frac{(-1)^t}{\binom{-\frac{3}{2}}{t}}>\frac{1}{4t^2}\ \ \text{for all}\ \ t\geq 1\Biggr)\nonumber\\
\hspace{-0.25 cm}&=&\hspace{-0.25 cm}  \Biggl(\frac{(-1)^t}{\binom{-\frac{3}{2}}{t}}-\frac{1}{4t^2}\Biggr)\Bigl(\cosh(\alpha)-1-\frac{\alpha^4}{18t^2}\Bigr)-\frac{1}{2t}\alpha\sinh(\alpha)-\frac{1}{2t^2}\Bigl(\frac{\alpha^2}{4}\cosh(\alpha)+\frac{\alpha}{4}\sinh(\alpha)\Bigr)\nonumber\\
\hspace{-0.25 cm}&&\hspace{-0.25 cm}  \hspace{10 cm}\ \ \Bigl(\text{by Lemma\ \ref{helperC2}}\Bigr)\nonumber\\
\hspace{-0.25 cm}&>&\hspace{-0.25 cm} \Biggl(\frac{(-1)^t}{\binom{-\frac{3}{2}}{t}}-\frac{1}{4t^2}\Biggr)\Bigl(\cosh(\alpha)-1\Bigr)-\frac{\alpha^4}{18t^2}-\frac{1}{2t}\alpha\sinh(\alpha)-\frac{1}{2t^2}\Bigl(\frac{\alpha^2}{4}\cosh(\alpha)+\frac{\alpha}{4}\sinh(\alpha)\Bigr)\nonumber\\
\hspace{-0.25 cm}&&\hspace{-0.25 cm} \hspace{7 cm} \Biggl(\text{as}\ \ \frac{(-1)^t}{\binom{-\frac{3}{2}}{t}}-\frac{1}{4t^2}<1\ \ \text{for all}\ \ t\geq 1\Biggr)\nonumber\\
\hspace{-0.25 cm}&=&\hspace{-0.25 cm} \frac{(-1)^t}{\binom{-\frac{3}{2}}{t}}(\cosh(\alpha)-1)-\frac{1}{2t}\alpha\sinh(\alpha)-\frac{1}{2t^2}\Biggl(\frac{\cosh(\alpha)-1}{2}+\frac{\alpha^4}{9}+\frac{\alpha^2}{4}\cosh(\alpha)+\frac{\alpha}{4}\sinh(\alpha)\Biggr).\nonumber\\
\hspace{-0.25 cm}&>&\hspace{-0.25 cm} \frac{(-1)^t}{\binom{-\frac{3}{2}}{t}}(\cosh(\alpha)-1)-\frac{1}{2t}\alpha\sinh(\alpha)-\frac{1}{8t^2}\nonumber\\
\hspace{-0.25 cm}&&\hspace{-0.25 cm} \hspace{4.5 cm} \Biggl(\text{as}\ \frac{\cosh(\alpha)-1}{2}+\frac{\alpha^4}{9}+\frac{\alpha^2}{4}\cosh(\alpha)+\frac{\alpha}{4}\sinh(\alpha)<\frac{1}{4} \Biggr).\nonumber\\
\end{eqnarray}
\endgroup 
Similarly, for the upper bound,  we have for all $t\geq 1$,
\begingroup
\allowdisplaybreaks
\begin{eqnarray}\label{S1lemprfeqn5}
\hspace{-0.25 cm}&&\hspace{-0.25 cm}\frac{S_1(t)}{(-1)^t\binom{-\frac{3}{2}}{t}}\nonumber\\
\hspace{-0.25 cm}&\leq &\hspace{-0.25 cm} \frac{(-1)^t}{\binom{-\frac{3}{2}}{t}}\sum_{u=1}^t\frac{\alpha^{2u}}{(2u)!}-\frac{1}{t}\sum_{u=1}^t\frac{u\alpha^{2u}}{(2u)!}+\frac{1}{4t^2}\sum_{u=1}^t\frac{\alpha^{2u}}{(2u)!}+\frac{2}{3t^2}\sum_{u=1}^t\frac{u\alpha^{2u}}{(2u)!}+\frac{1}{2t^2}\sum_{u=1}^t\frac{u^2\alpha^{2u}}{(2u)!}+\frac{1}{3t^2}\sum_{u=1}^t\frac{u^3\alpha^{2u}}{(2u)!}\nonumber\\
\hspace{-0.25 cm}&=&\hspace{-0.25 cm}\frac{(-1)^t}{\binom{-\frac{3}{2}}{t}}\sum_{u=1}^{\infty}\frac{\alpha^{2u}}{(2u)!}-\frac{1}{t}\sum_{u=0}^{\infty}\frac{u\alpha^{2u}}{(2u)!}+\frac{1}{t}\sum_{u=t+1}^{\infty}\frac{u\alpha^{2u}}{(2u)!}+\frac{1}{4t^2}\sum_{u=0}^{\infty}\frac{\alpha^{2u}}{(2u)!}\nonumber\\
\hspace{-0.25 cm}&&\hspace{-0.25 cm} \hspace{5 cm} +\frac{2}{3t^2}\sum_{u=0}^{\infty}\frac{u\alpha^{2u}}{(2u)!}+\frac{1}{2t^2}\sum_{u=0}^{\infty}\frac{u^2\alpha^{2u}}{(2u)!}+\frac{1}{3t^2}\sum_{u=0}^{\infty}\frac{u^3\alpha^{2u}}{(2u)!}\nonumber\\
\hspace{-0.25 cm}&\leq &\hspace{-0.25 cm} \frac{(-1)^t}{\binom{-\frac{3}{2}}{t}}(\cosh(\alpha)-1)-\frac{1}{2t}\alpha\sinh(\alpha)+\frac{\alpha^4}{9t^3}+\frac{1}{4t^2}\cosh(\alpha)+\frac{1}{3t^2}\alpha\sinh(\alpha)\nonumber\\
\hspace{-0.25 cm}&&\hspace{-0.25 cm} \hspace{2 cm} +\frac{1}{2t^2}\Biggl(\frac{\alpha^2}{4}\cosh(\alpha)+\frac{\alpha}{4}\sinh(\alpha)\Biggr)+\frac{1}{3t^2}\Biggl(\frac{3\alpha^2}{8}\cosh(\alpha)+\frac{\alpha(\alpha^2+1)}{8}\sinh(\alpha)\Biggr)\nonumber\\
\hspace{-0.25 cm}&&\hspace{-0.25 cm} \hspace{9.5 cm} \Bigl(\text{by Lemmas\ \ref{helperC2}\ and\ \ \ \ref{helperC}}\Bigr)\nonumber\\
\hspace{-0.25 cm}&=&\hspace{-0.25 cm} \frac{(-1)^t}{\binom{-\frac{3}{2}}{t}}(\cosh(\alpha)-1)-\frac{1}{2t}\alpha\sinh(\alpha)+\frac{1}{t^2}\Biggl(\frac{\alpha^4}{9t}+\frac{\alpha^2+1}{4}\cosh(\alpha)+\frac{\alpha(\alpha^2+12)}{24}\sinh(\alpha)\Biggr)\nonumber\\
\hspace{-0.25 cm}&\leq&\hspace{-0.25 cm} \frac{(-1)^t}{\binom{-\frac{3}{2}}{t}}(\cosh(\alpha)-1)-\frac{1}{2t}\alpha\sinh(\alpha)+\frac{1}{t^2}\Biggl(\frac{\alpha^4}{9}+\frac{\alpha^2+1}{4}\cosh(\alpha)+\frac{\alpha(\alpha^2+12)}{24}\sinh(\alpha)\Biggr)\nonumber\\
\hspace{-0.25 cm}&<&\hspace{-0.25 cm} \frac{(-1)^t}{\binom{-\frac{3}{2}}{t}}(\cosh(\alpha)-1)-\frac{1}{2t}\alpha\sinh(\alpha)+\frac{13}{25 t^2}\nonumber\\
\hspace{-0.25 cm}&&\hspace{-0.25 cm} \hspace{5 cm} \Biggl(\text{as}\ \frac{\alpha^4}{9}+\frac{\alpha^2+1}{4}\cosh(\alpha)+\frac{\alpha(\alpha^2+12)}{24}\sinh(\alpha)<\frac{13}{25} \Biggr).
\end{eqnarray}
\endgroup 
By \eqref{S1lemprfeqn4} and \eqref{S1lemprfeqn5}, for all $t\geq 1$, it follows that
\begin{equation}\label{sum1final}
-\frac{1}{8t^2}< \frac{S_1(t)}{(-1)^t\binom{-\frac{3}{2}}{t}}-\frac{(-1)^t}{\binom{-\frac{3}{2}}{t}}(\cosh(\alpha)-1)+\frac{1}{2t}\alpha\sinh(\alpha)< \frac{13}{25t^2},
\end{equation}
which concludes the proof.
\qed

\emph{Proof of Lemma \ref{S2lemma}:}
	Rewrite $S_2(t)$ as follows:
	\begin{eqnarray}\label{S2lemprfeqn1}
	S_2(t)&=&\sum_{u=0}^{t-1}\frac{(-1)^u\alpha^{2u}}{(2u)!}\sum_{s=u}^{t-1}\Bigl(\frac{1}{2} - s\Bigr)_{s + 1} \binom{-\frac{3}{2}}{t - s - 1}\frac{(-s)_u}{(s + u + 1)!}\nonumber\\
	&=&\sum_{u=0}^{t-1}\frac{(-1)^u\alpha^{2u}}{(2u)!}\underbrace{\sum_{s=0}^{t-u-1}\Bigl(\frac{1}{2} - s-u\Bigr)_{s+u+1} \binom{-\frac{3}{2}}{t - s - u - 1}\frac{(-s-u)_u}{(s + 2u + 1)!}}_{=:S_2(t,u)}\nonumber\\
	\end{eqnarray}
	Using the summation package \texttt{Sigma} (and its mechanization by \texttt{EvaluateMultiSums})\footnote[3]{We refer again to Appendix \ref{reductionbySigma} and Remark \ref{remarkappendix} to see the underlying machinery in action.} we derive and prove that
	\begin{equation}\label{S2lemprfeqn2}
	S_2(t,u)=\binom{-\frac{3}{2}}{t}(-1)^{u+1} \Bigl(A_{2,1}(t,u)+A_{2,2}(t,u)\Bigr),
	\end{equation}
	where
	\begin{equation*}
	\hspace{-2.2 cm}
	A_{2,1}(t,u)=\frac{2t(t-u)(-t)_u(-1)^{u}}{(1+2t)(1+2u)(t+u)(t)_u}
	\end{equation*}
	and
	\begin{equation*}
	\hspace{-0.5 cm}A_{2,2}(t,u)=\frac{(-1)^{t+1}}{\binom{-\frac{3}{2}}{t}}+\frac{1}{1+2t}+\frac{2t}{1+2t}\sum_{i=1}^u\frac{(-1)^i(-t)_i}{(t+i)(t)_i}.
	\end{equation*}
	From \eqref{S2lemprfeqn1} and \eqref{S2lemprfeqn2} it follows that
	\begin{equation}\label{S2lemprfeqn3}
	S_2(t)=-\binom{-\frac{3}{2}}{t}\Bigl(s_{2,1}(t)+s_{2,2}(t)\Bigr),
	\end{equation}
	where
	\begin{equation}\label{S2lemprfeqn4}
	s_{2,1}(t)=\sum_{u=0}^{t-1}\frac{\alpha^{2u}}{(2u)!}A_{2,1}(t,u)\ \ \text{and}\ \ s_{2,2}(t)=\sum_{u=0}^{t-1}\frac{\alpha^{2u}}{(2u)!}A_{2,2}(t,u).
	\end{equation}
	By Lemma \ref{lem2}, we have
	\begin{equation}\label{sum2part1}
	\frac{1}{1+2u}-\frac{u^2+u+\frac{1}{2}}{t(1+2u)}\leq \frac{t-u}{t(1+2u)}\Biggl(1-\frac{u^2+\frac{1}{2}}{t}\Biggr)\leq A_{2,1}(t,u)\leq \frac{t-u}{t(1+2u)} .
	\end{equation}
	Plugging \eqref{sum2part1} into \eqref{S2lemprfeqn4} we obtain
	\begin{equation*}
	\sum_{u=0}^{t-1}\frac{\alpha^{2u}}{(2u+1)!}-\frac{1}{t}\sum_{u=0}^{t-1}\frac{u^2+u+\frac{1}{2}}{(2u+1)!}\alpha^{2u} \leq s_{2,1}(t) \leq \sum_{u=0}^{t-1}\frac{\alpha^{2u}}{(2u+1)!}-\frac{1}{t}\sum_{u=0}^{t-1}\frac{u\alpha^{2u}}{(2u+1)!},
	\end{equation*}
	and consequently, 
	\begin{equation}\label{s21rewr}
	\begin{split}
	&\sum_{u=0}^{\infty}\frac{\alpha^{2u}}{(2u+1)!}-\sum_{u=t}^{\infty}\frac{\alpha^{2u}}{(2u+1)!}-\frac{1}{t}\sum_{u=0}^{\infty}\frac{u^2+u+\frac{1}{2}}{(2u+1)!}\alpha^{2u} \leq s_{2,1}(t)\leq \\
	& \hspace{5 cm}  \sum_{u=0}^{\infty}\frac{\alpha^{2u}}{(2u+1)!}-\frac{1}{t}\Biggl(\sum_{u=0}^{\infty} \frac{u\alpha^{2u}}{(2u+1)!}-\sum_{u=t}^{\infty}\frac{u\alpha^{2u}}{(2u+1)!}\Biggr).
	\end{split}
	\end{equation}
	By Lemma \ref{helperC}, 
	\begin{equation}\label{s21h}
	\sum_{u=t}^{\infty}\frac{\alpha^{2u}}{(2u+1)!}=\frac{1}{\alpha^2}\sum_{u=t+1}^{\infty}\frac{\alpha^{2u}}{(2u-1)!}=\frac{2}{\alpha^2}\sum_{u=t+1}^{\infty}\frac{u\alpha^{2u}}{(2u)!}\leq \frac{2C_1}{t^2}=\dfrac{2\alpha^2}{9 t^2},
	\end{equation}
	and
	\begin{equation}\label{s21h2}
	\sum_{u=t}^{\infty}\frac{u\alpha^{2u}}{(2u+1)!}=\frac{2}{\alpha^2}\sum_{u=t+1}^{\infty}\frac{u(u-1)\alpha^{2u}}{(2u)!}\leq\frac{2}{\alpha^2}\sum_{u=t+1}^{\infty}\frac{u^2\alpha^{2u}}{(2u)!}\leq \frac{2C_2}{\alpha^2 t^2}=\frac{4\alpha^2}{9t^2}.
	\end{equation}
	Plugging \eqref{s21h} and \eqref{s21h2} into \eqref{s21rewr} gives
	\begin{equation}\label{S2lemprfeqn5}
	\sum_{u=0}^{\infty}\frac{\alpha^{2u}}{(2u+1)!}-\frac{2\alpha^2}{9 t^2}-\frac{1}{t}\sum_{u=0}^{\infty}\frac{u^2+u+\frac{1}{2}}{(2u+1)!}\alpha^{2u} \leq s_{2,1}(t) \leq \sum_{u=0}^{\infty}\frac{\alpha^{2u}}{(2u+1)!}-\frac{1}{t}\sum_{u=0}^{\infty}\frac{u\alpha^{2u}}{(2u+1)!}+\frac{4\alpha^2}{9 t^3}.
	\end{equation}
	Using Lemma \ref{helperC2}, \eqref{S2lemprfeqn5} further reduces to 
	\begin{equation}\label{S2lemprfeqn6}
	\begin{split}
	&\frac{\sinh(\alpha)}{\alpha}-\frac{1}{t}\Biggl(\frac{\cosh(\alpha)}{4}+\frac{\sinh(\alpha)}{4\alpha}+\frac{\alpha\sinh(\alpha)}{4}+\frac{2\alpha^2}{9}\Biggr)\leq s_{2,1}(t)\leq\\
	& \hspace{6 cm} \frac{\sinh(\alpha)}{\alpha}-\frac{1}{t}\Biggl(\frac{\cosh(\alpha)}{2}-\frac{\sinh(\alpha)}{2\alpha}-\frac{4\alpha^2}{9}\Biggr).
	\end{split}
	\end{equation}
	A numerical check shows that 
	$$\frac{\cosh(\alpha)}{4}+\frac{\sinh(\alpha)}{4\alpha}+\frac{\alpha\sinh(\alpha)}{4}+\frac{2\alpha^2}{9}<\frac{7}{10}\ \ \text{and}\ \ \frac{\cosh(\alpha)}{2}-\frac{\sinh(\alpha)}{2\alpha}-\frac{4\alpha^2}{9}>-\frac{3}{40}.$$
	This, along with \eqref{S2lemprfeqn6}, gives
	\begin{equation}\label{finals21}
	\frac{\sinh(\alpha)}{\alpha}-\frac{7}{10t}  < s_{2,1}(t)<\frac{\sinh(\alpha)}{\alpha}+\frac{3}{40t}.
	\end{equation}
Next we employ Lemma \ref{lem3} and get
	\begin{equation}\label{a22final}
	\frac{2u+1}{2t}-\frac{4u^3+6u^2+8u+3}{12t^2}+\frac{(-1)^{t+1}}{\binom{-\frac{3}{2}}{t}} \leq A_{2,2}(t,u) \leq \frac{2u+1}{2t}+\frac{(-1)^{t+1}}{\binom{-\frac{3}{2}}{t}}.
	\end{equation}
	Plugging \eqref{a22final} into \eqref{S2lemprfeqn4}, we obtain
	\begin{equation*}
	\sum_{u=0}^{t-1}\frac{\alpha^{2u}}{(2u)!}\Biggl(\frac{2u+1}{2t}+\frac{(-1)^{t+1}}{\binom{-\frac{3}{2}}{t}}-\frac{4u^3+6u^2+8u+3}{12t^2}\Biggr) \leq s_{2,2}(t) \leq \sum_{u=0}^{t-1}\frac{\alpha^{2u}}{(2u)!}\Biggl(\frac{2u+1}{2t}+\frac{(-1)^{t+1}}{\binom{-\frac{3}{2}}{t}}\Biggr),
	\end{equation*}
	which, using $p_3(u):=4u^3+6u^2+8u+3$, can be rewritten as
	\begin{equation}\label{s22interm}
	\begin{split}
	& \frac{1}{2t}\sum_{u=0}^{\infty}\frac{(2u+1)\alpha^{2u}}{(2u)!}-\frac{1}{2t}\sum_{u=t}^{\infty}\frac{(2u+1)\alpha^{2u}}{(2u)!}+\frac{(-1)^{t+1}}{\binom{-\frac{3}{2}}{t}}\sum_{u=0}^{\infty}\frac{\alpha^{2u}}{(2u)!}-\frac{1}{12t^2}\sum_{u=0}^{\infty}\frac{p_3(u)\alpha^{2u}}{(2u)!}\\
	& \hspace{4 cm} \leq s_{2,2}(t)\leq \frac{1}{2t}\sum_{u=0}^{\infty}\frac{(2u+1)\alpha^{2u}}{(2u)!}+\frac{(-1)^{t+1}}{\binom{-\frac{3}{2}}{t}}\sum_{u=0}^{\infty}\frac{\alpha^{2u}}{(2u)!}-\frac{(-1)^{t+1}}{\binom{-\frac{3}{2}}{t}}\sum_{u=t}^{\infty}\frac{\alpha^{2u}}{(2u)!}.
	\end{split}
	\end{equation}
	By Lemma \ref{helperC} we obtain
	\begin{equation}\label{s22trun1}
	\sum_{u=t}^{\infty}\frac{\alpha^{2u}}{(2u)!}=\frac{1}{\alpha^2}\sum_{u=t+1}^{\infty}\frac{(2u-1)2u\alpha^{2u}}{(2u)!}\leq \frac{4}{\alpha^2}\sum_{u=t+1}^{\infty}\frac{u^2\alpha^{2u}}{(2u)!}\leq \frac{4C_2}{\alpha^2 t^2}=\frac{8\alpha^2}{9 t^2}
	\end{equation}
	and
	\begin{equation}\label{s22trun2}
	\sum_{u=t}^{\infty}\frac{(2u+1)\alpha^{2u}}{(2u)!}=\frac{1}{\alpha^2}\sum_{u=t+1}^{\infty}
	\frac{2u(2u-1)^2\alpha^{2u}}{(2u)!}\leq \frac{8}{\alpha^2}\sum_{u=t+1}^{\infty}\frac{u^3\alpha^{2u}}{(2u)!}\leq \frac{8C_3}{\alpha^2 t^2}=\frac{32\alpha^2}{9 t^2}.
	\end{equation}
	Combining \eqref{s22trun1} and \eqref{s22trun2} with \eqref{s22interm} gives
	\begin{equation}\label{S2lemprfeqn7}
	\begin{split}
	&\frac{1}{2t}\sum_{u=0}^{\infty}\frac{(2u+1)\alpha^{2u}}{(2u)!}-\frac{1}{2t}\frac{32\alpha^2}{9 t^2}+\frac{(-1)^{t+1}}{\binom{-\frac{3}{2}}{t}}\sum_{u=0}^{\infty}\frac{\alpha^{2u}}{(2u)!}-\frac{1}{12t^2}\sum_{u=0}^{\infty}\frac{p_3(u)\alpha^{2u}}{(2u)!}\\
	& \hspace{3.4 cm} \leq s_{2,2}(t)\leq \frac{1}{2t}\sum_{u=0}^{\infty}\frac{(2u+1)\alpha^{2u}}{(2u)!}+\frac{(-1)^{t+1}}{\binom{-\frac{3}{2}}{t}}\sum_{u=0}^{\infty}\frac{\alpha^{2u}}{(2u)!}+\frac{(-1)^{t}}{\binom{-\frac{3}{2}}{t}}\frac{8\alpha^2}{9t^2}.
	\end{split}
	\end{equation}
	Furthermore, for all $t\geq 1$ we have $\binom{2t}{t}\geq \frac{4^t}{2\sqrt{t}}$ which implies
	\begin{equation}\label{S2lemprfeqn8}
	\frac{(-1)^t}{\binom{-\frac{3}{2}}{t}}=\frac{2^{2t+1}}{t+1}\frac{1}{\binom{2t+2}{t+1}}<1,\ \ \text{$t\geq 1$.}
	\end{equation}
	Applying \eqref{S2lemprfeqn8} and Lemma \ref{helperC} to \eqref{S2lemprfeqn7}, we obtain
	\begin{equation}\label{S2lemprfeqn9}
	\begin{split} 
	&\frac{(-1)^{t+1}}{\binom{-\frac{3}{2}}{t}}\cosh(\alpha)+\frac{1}{2t}\Bigl(\underbrace{\cosh(\alpha)+\alpha\sinh(\alpha)}_{=:\mathrm{csh}(\alpha)}\Bigr)-\frac{C_{2,2}(\alpha)}{t^2}\leq s_{2,2}(t)\leq\\
	& \hspace{4 cm} \frac{(-1)^{t+1}}{\binom{-\frac{3}{2}}{t}}\cosh(\alpha)+\frac{1}{2t}\mathrm{csh}(\alpha)+\frac{8\alpha^2}{9},
	\end{split}
	\end{equation}
	where 
	$$C_{2,2}(\alpha)=\frac{16\alpha^2}{9}+\frac{\alpha^2\cosh(\alpha)}{4}+\frac{\alpha^{3}\sinh(\alpha)}{24}+\frac{\cosh(\alpha)}{4}+\frac{\alpha\sinh(\alpha)}{2}<1\ \ \text{and}\ \ \frac{8\alpha^2}{9}<\frac{1}{4}.$$
	Therefore
	\begin{equation}\label{s22final}
	\begin{split}
	& \frac{(-1)^{t+1}}{\binom{-\frac{3}{2}}{t}}\cosh(\alpha)+\frac{1}{2t}\mathrm{csh}(\alpha)-\frac{1}{t^2} \leq s_{2,2}(t)\leq \\
	& \hspace{3 cm}  \frac{(-1)^{t+1}}{\binom{-\frac{3}{2}}{t}}\cosh(\alpha)+\frac{1}{2t}\mathrm{csh}(\alpha)+\frac{1}{4t^2}.
	\end{split}
	\end{equation}
	Applying \eqref{finals21} and \eqref{s22final} to \eqref{S2lemprfeqn3} we obtain
	\begin{equation*}
	\begin{split}
	&\frac{\sinh(\alpha)}{\alpha}-\frac{7}{10t}+\frac{(-1)^{t+1}}{\binom{-\frac{3}{2}}{t}}\cosh(\alpha)+\frac{1}{2t}\mathrm{csh}(\alpha)-\frac{1}{t^2} \leq -\frac{S_2(t)}{\binom{\frac{-3}{2}}{t}}\leq\\
	& \hspace{2 cm} \frac{\sinh(\alpha)}{\alpha}+\frac{3}{40t}+\frac{(-1)^{t+1}}{\binom{-\frac{3}{2}}{t}}\cosh(\alpha)+\frac{1}{2t}\mathrm{csh}(\alpha)+\frac{1}{4t^2},
	\end{split}
	\end{equation*}
	which implies that for $t \geq 1$,
	\begin{equation}\label{S2lemprfeqn10}
	\begin{split}
	&\frac{\sinh(\alpha)}{\alpha}+\frac{(-1)^{t+1}}{\binom{-\frac{3}{2}}{t}}\cosh(\alpha)+\frac{1}{t}\Biggl(-\frac{7}{10}+\frac{\mathrm{csh}(\alpha)}{2}-1\Biggr)\leq -\frac{S_2(t)}{\binom{\frac{-3}{2}}{t}}\leq\\
	& \frac{\sinh(\alpha)}{\alpha}+\frac{(-1)^{t+1}}{\binom{-\frac{3}{2}}{t}}\cosh(\alpha)+\frac{1}{t}\Biggl(\frac{3}{40}+\frac{\mathrm{csh}(\alpha)}{2}+\frac{1}{4}\Biggr).
	\end{split}
	\end{equation}
	Since
	$$-\frac{7}{10}+\frac{\mathrm{csh}(\alpha)}{2}-1>-1 \ \ \text{and}\ \ \frac{3}{40}+\frac{\mathrm{csh}(\alpha)}{2}+\frac{1}{4}<\frac{11}{10},$$
	from \eqref{S2lemprfeqn10}, it follows that for all $t\geq 1$,
	\begin{equation}\label{sum2final}
	-\frac{1}{t}< -\frac{\sinh(\alpha)}{\alpha}-\frac{(-1)^{t+1}}{\binom{-\frac{3}{2}}{t}}\cosh(\alpha)-\frac{S_2(t)}{\binom{-\frac{3}{2}}{t}}< \frac{11}{10 t}.
	\end{equation}
	Multiplying by $-1$ on both sides of \eqref{sum2final}, we get \eqref{S2lemmaeqn}.
\qed

\emph{Proof of Lemma \ref{S3lemma}:}
Rewrite $S_3(t)$ as follows:
\begin{eqnarray}\label{S3lemprfeqn1}
S_3(t)&=& \sum_{u=1}^t\frac{(-1)^u\alpha^{2u}}{(2u-1)!}\sum_{s=u}^t\frac{1}{s}\Bigl(\frac{1}{2}-s\Bigr)_{s+1}\binom{-\frac{3}{2}}{t-s}\frac{(-s)_u}{(s+u)!}\nonumber\\
&=& \sum_{u=1}^t\frac{(-1)^u\alpha^{2u}}{(2u-1)!}\underbrace{\sum_{s=0}^{t-u}\frac{1}{s+u}\Bigl(\frac{1}{2}-s-u\Bigr)_{s+u+1}\binom{-\frac{3}{2}}{t-s-u}\frac{(-s-u)_u}{(s+2u)!}}_{=:S_3(t,u)}\nonumber\\
\end{eqnarray}
Using the summation package \texttt{Sigma} (and its mechanization by \texttt{EvaluateMultiSums}), the sum $S_3(t,u)$ can be rewritten\footnote[4]{We refer again to Appendix \ref{reductionbySigma} and Remark \ref{remarkappendix} to see the underlying machinery in action.} as an \textit{indefinite} sum
\begin{equation}\label{S3lemprfeqn2}
S_3(t,u)=\binom{-\frac{3}{2}}{t}(-1)^u \Bigl(A_{3,1}(t,u)+A_{3,2}(t,u)\Bigr),
\end{equation}
where 
\begin{equation*}
A_{3,1}(t,u)=\frac{t(1+2t-2u)(-t)_u(-1)^u}{2(1+2t)u(t+u)(t)_u}
\end{equation*}
and 
\begin{equation*}
\hspace{2.5 cm} A_{3,2}(t,u)=\frac{(-1)^{t+1}}{\binom{-\frac{3}{2}}{t}}+\frac{1}{1+2t}+\frac{2t}{1+2t}\sum_{i=1}^u\frac{(-t)_i(-1)^i}{(t+i)(t)_i}.
\end{equation*}
From \eqref{S3lemprfeqn1} and \eqref{S3lemprfeqn2}, it follows that
\begin{equation}\label{S3lemprfeqn3}
S_3(t)=\binom{-\frac{3}{2}}{t}\Bigl(s_{3,1}(t)+s_{3,2}(t)\Bigr),
\end{equation}
where
\begin{equation}\label{S3lemprfeqn4}
s_{3,1}(t)=\sum_{u=1}^t\frac{\alpha^{2u}}{(2u-1)!}A_{3,1}(t,u)\ \ \text{and}\ \ s_{3,2}(t)=\sum_{u=1}^t\frac{\alpha^{2u}}{(2u-1)!}A_{3,2}(t,u).
\end{equation}
By Lemma \ref{lem2}, we have
\begin{equation}\label{S3lemprfeqn5}
-\frac{1+2t-2u}{2u}\frac{1}{2t}\frac{u^2+\frac{1}{2}}{t}\leq A_{3,1}(t,u)-\frac{1+2t-2u}{2u}\frac{1}{2t}=A_{3,1}(t,u)-\frac{1}{2u}+\frac{2u-1}{4ut} \leq 0.
\end{equation}
Equation \eqref{S3lemprfeqn5} implies that
\begin{equation}\label{a31final}
-\frac{3u^2+2u+\frac{1}{2}}{4ut}=-\frac{u^2+\frac{1}{2}}{2ut}-\frac{\frac{u^2}{2}+\frac{1}{4}}{2ut}-\frac{2u-1}{4ut}\leq A_{3,1}(t,u)-\frac{1}{2u} \leq -\frac{2u-1}{4ut} \leq 0.
\end{equation}
Plugging \eqref{a31final} into \eqref{S3lemprfeqn4}, we obtain
\begin{equation*}
-\frac{1}{2t}\sum_{u=1}^{\infty}\frac{(3u^2+2u+\frac{1}{2})\alpha^{2u}}{(2u)!}\leq -\frac{1}{2t}\sum_{u=1}^t\frac{(3u^2+2u+\frac{1}{2})\alpha^{2u}}{(2u)!}\leq s_{3,1}(t)-\sum_{u=1}^t\frac{\alpha^{2u}}{(2u)!} \leq 0,
\end{equation*}
and consequently,
\begin{equation}\label{S3lemprfeqn6}
-\frac{1}{2t}\sum_{u=1}^{\infty}\frac{(3u^2+2u+\frac{1}{2})\alpha^{2u}}{(2u)!}-\sum_{u=t+1}^{\infty}\frac{\alpha^{2u}}{(2u)!}\leq s_{3,1}(t)-\sum_{u=1}^{\infty}\frac{\alpha^{2u}}{(2u)!}\leq -\sum_{u=t+1}^{\infty}\frac{\alpha^{2u}}{(2u)!} \leq 0.
\end{equation}
Applying Lemmas \ref{helperC} and \ref{helperC2} to \eqref{S3lemprfeqn6} gives
\begin{equation}\label{s31final}
-\frac{1}{2t}<-\frac{1}{t}\Biggl(\frac{3\alpha^2\cosh(\alpha)+7\alpha\sinh(\alpha)+2\cosh(\alpha)-2}{8}+\frac{\alpha^4}{9}\Biggr)\leq s_{3,1}(t)+1-\cosh(\alpha) \leq 0.
\end{equation}
Next, by Lemma \ref{lem3}, we obtain
\begin{equation}\label{a32final}
-\frac{4u^3+6u^2+8u+3}{12t^2}\leq A_{3,2}(t,u)+\frac{(-1)^{t}}{\binom{-\frac{3}{2}}{t}}-\frac{2u+1}{2t} \leq 0.
\end{equation}
Applying \eqref{a32final} to \eqref{S3lemprfeqn4}, it follows that
\begin{equation}\label{S3lemprfeqn7}
 s_{3,2}(t)+\frac{(-1)^{t}}{\binom{-\frac{3}{2}}{t}}\sum_{u=1}^t\frac{\alpha^{2u}}{(2u-1)!}-\frac{1}{2t}\sum_{u=1}^t\frac{(2u+1)\alpha^{2u}}{(2u-1)!}\leq 0
\end{equation}
and
\begin{equation}\label{S3lemprfeqn8}
\begin{split}
s_{3,2}(t)+\frac{(-1)^{t}}{\binom{-\frac{3}{2}}{t}}\sum_{u=1}^t\frac{\alpha^{2u}}{(2u-1)!}-\frac{1}{2t}\sum_{u=1}^t\frac{(2u+1)\alpha^{2u}}{(2u-1)!}&\geq -\frac{1}{12t^2}\sum_{u=1}^t\frac{p_3(u)\alpha^{2u}}{(2u-1)!}\\
&\geq -\frac{1}{12t^2}\sum_{u=1}^{\infty}\frac{(p_3(u)\alpha^{2u}}{(2u-1)!},
\end{split}
\end{equation}
where $p_3(u)=4u^3+6u^2+8u+3$ is as in \eqref{s22interm}.
Equations \eqref{S3lemprfeqn7} and \eqref{S3lemprfeqn8} imply that
\begin{equation}\label{s32interm1}
s_{3,2}(t)+\frac{(-1)^{t}}{\binom{-\frac{3}{2}}{t}}\sum_{u=1}^{\infty}\frac{\alpha^{2u}}{(2u-1)!}-\frac{1}{2t}\sum_{u=1}^{\infty}\frac{(2u+1)\alpha^{2u}}{(2u-1)!} \leq \frac{(-1)^{t}}{\binom{-\frac{3}{2}}{t}}\sum_{u=t+1}^{\infty}\frac{\alpha^{2u}}{(2u-1)!},
\end{equation}
and
\begin{equation}\label{s32interm2}
\begin{split}
s_{3,2}(t)+\frac{(-1)^{t}}{\binom{-\frac{3}{2}}{t}}\sum_{u=1}^{\infty}\frac{\alpha^{2u}}{(2u-1)!}&-\frac{1}{2t}\sum_{u=1}^{\infty}\frac{(2u+1)\alpha^{2u}}{(2u-1)!} \geq \\
& -\frac{1}{12t^2}\sum_{u=1}^{\infty}\frac{p_3(u)\alpha^{2u}}{(2u-1)!}-\frac{1}{2t}\sum_{u=t+1}^{\infty}\frac{(2u+1)\alpha^{2u}}{(2u-1)!}.
\end{split}
\end{equation}
By Lemma \ref{helperC} we obtain
\begin{equation}\label{S3lemprfeqn9}
\sum_{u=t+1}^{\infty}\frac{\alpha^{2u}}{(2u-1)!}=2\sum_{u=t+1}^{\infty}\frac{u\alpha^{2u}}{(2u)!}\leq \frac{4\alpha^4}{3\cdot 3! t^2}=\frac{2\alpha^4}{9 t^2}
\end{equation}
and
\begin{equation}\label{S3lemprfeqn10}
\sum_{u=t+1}^{\infty}\frac{(2u+1)\alpha^{2u}}{(2u-1)!}=2u\sum_{u=t+1}^{\infty}\frac{(2u+1)\alpha^{2u}}{(2u)!}\leq \frac{20\alpha^4}{3\cdot 3!t^2}=\frac{10\alpha^4}{9 t^2}.
\end{equation}
Substituting \eqref{S3lemprfeqn9}-\eqref{S3lemprfeqn10} into \eqref{s32interm1} and \eqref{s32interm2}, it follows that
\begin{equation}\label{S3lemprfeqn11}
s_{3,2}(t)+\frac{(-1)^{t}}{\binom{-\frac{3}{2}}{t}}\sum_{u=1}^{\infty}\frac{\alpha^{2u}}{(2u-1)!}-\frac{1}{2t}\sum_{u=1}^{\infty}\frac{(2u+1)\alpha^{2u}}{(2u-1)!} \leq \frac{3}{2}\cdot\frac{2\alpha^4}{9 t^2}=\frac{\alpha^4}{3 t^2}
\end{equation}
and
\begin{equation}\label{S3lemprfeqn12}
\begin{split}
-\frac{1}{12t^2}\sum_{u=1}^{\infty}\frac{p_3(u)\alpha^{2u}}{(2u-1)!}-\frac{\alpha^4}{3t^2}& \leq -\frac{1}{12t^2}\sum_{u=1}^{\infty}\frac{p_3(u)\alpha^{2u}}{(2u-1)!}-\frac{1}{2t}\frac{10\alpha^4}{9 t^2}\leq\\
& s_{3,2}(t)+\frac{(-1)^{t}}{\binom{-\frac{3}{2}}{t}}\sum_{u=1}^{\infty}\frac{\alpha^{2u}}{(2u-1)!}-\frac{1}{2t}\sum_{u=1}^{\infty}\frac{(2u+1)\alpha^{2u}}{(2u-1)!}.
\end{split}
\end{equation} 
Using Lemma \ref{helperC2} into \eqref{S3lemprfeqn11} and \eqref{S3lemprfeqn12}, we obtain
\begin{equation}\label{s32final}
\begin{split}
-\frac{61}{100t^2}&< -\frac{1}{t^2}\Biggl(\frac{3\alpha^{3}\sinh(\alpha)}{8}+\frac{(\alpha^4+24\alpha^2)\cosh(\alpha)}{24}+\frac{3\alpha\sinh(\alpha)}{4}+\frac{5\alpha^4}{9}\Biggr)\leq \\
&s_{3,2}(t)+\frac{(-1)^{t}}{\binom{-\frac{3}{2}}{t}}\alpha\sinh(\alpha)-\frac{1}{2t}\mathrm{sch}(\alpha) \leq \frac{\alpha^4}{3 t^2}< \frac{3}{100 t^2},
\end{split}
\end{equation}
where $\mathrm{sch}(\alpha):=\alpha^2 \cosh(\alpha)+2\alpha\sinh(\alpha)$.
Combining \eqref{s31final} and \eqref{s32final}, and then plugging into \eqref{S3lemprfeqn3} it follows that
\begin{equation*}
-\frac{1}{2t}-\frac{61}{100t^2}<\frac{S_3(t)}{\binom{-\frac{3}{2}}{t}}+\frac{(-1)^{t}}{\binom{-\frac{3}{2}}{t}}\alpha\sinh(\alpha)-\frac{1}{2t}\mathrm{sch}(\alpha)+1-\cosh(\alpha)<\frac{3}{100 t^2}.
\end{equation*}
Since for $t\geq 2$,
\begin{equation*}
-\frac{1}{2t}-\frac{61}{100t^2}+\frac{1}{2t}\mathrm{sch}(\alpha)>-\frac{71}{100t},
\end{equation*}
and
\begin{equation*}
\frac{3}{100 t^2}+\frac{1}{2t}\mathrm{sch}(\alpha)<\frac{12}{25t},
\end{equation*}
we finally get
\begin{equation}\label{sum3final}
\frac{12}{25t}> \frac{S_3(t)}{\binom{-\frac{3}{2}}{t}}+\frac{(-1)^{t}}{\binom{-\frac{3}{2}}{t}}\alpha\sinh(\alpha)+1-\cosh(\alpha)> -\frac{71}{100t}.
\end{equation}
\qed

\emph{Proof of Lemma \ref{S4lemma}:}
Rewrite $S_4(t)$ as follows:
\begin{eqnarray}\label{S4lemprfeqn1}
S_4(t)&=&\sum_{u=0}^t\frac{(-1)^u\alpha^{2u}}{(2u)!}\sum_{s=u}^t(-1)^s \Bigl(\frac{1}{2} - s\Bigr)_{s + 1} \frac{(-s)_u}{(s + u + 1)!}\nonumber\\
&=& \sum_{u=0}^t\frac{(-1)^u\alpha^{2u}}{(2u)!}\underbrace{\sum_{s=0}^{t-u}(-1)^{s+u} \Bigl(\frac{1}{2} - s-u\Bigr)_{s+u + 1} \frac{(-s-u)_u}{(s + 2u + 1)!}}_{=:S_4(t,u)}\nonumber\\
\end{eqnarray}
Using again the summation package \texttt{Sigma} (and its mechanization by \texttt{EvaluateMultiSums})\footnote[5]{We refer again to Appendix \ref{reductionbySigma} and Remark \ref{remarkappendix} to see the underlying machinery in action.}, we rewrite $S_4(t,u)$ as an indefinite sum
\begin{equation}\label{S4lemprfeqn2}
S_4(t,u)=\binom{-\frac{3}{2}}{t}(-1)^{u+t}\Bigl(A_{4,1}(t,u)+A_{4,2}(t,u)\Bigr),
\end{equation}
where
\begin{equation*}
A_{4,1}(t,u)=\frac{t(-t)_u(-1)^u}{2(1+2t)(t+u)(t+u+1)(t)_u}
\end{equation*}
and
\begin{equation*}
\hspace{3 cm} A_{4,2}(t,u)=\frac{1}{1+2u}\Biggl(\frac{(-1)^{t}}{\binom{-\frac{3}{2}}{t}}-\frac{1}{1+2t}-\frac{2t}{1+2t}\sum_{i=1}^u\frac{(-1)^i(-t)_i}{(t+i)(t)_i}\Biggr).
\end{equation*}
From \eqref{S4lemprfeqn1} and \eqref{S4lemprfeqn2} it follows that
\begin{equation}\label{S4lemprfeqn3}
S_4(t)=(-1)^t\binom{-\frac{3}{2}}{t}\Bigl(s_{4,1}(t)+s_{4,2}(t)\Bigr),
\end{equation}
where
\begin{equation}\label{S4lemprfeqn4}
s_{4,1}(t)=\sum_{u=0}^t\frac{\alpha^{2u}}{(2u)!}A_{4,1}(t,u)\ \ \text{and}\ \ s_{4,2}(t):=\sum_{u=0}^t\frac{\alpha^{2u}}{(2u)!}A_{4,2}(t).
\end{equation}
From Lemmas \ref{lem1}-\ref{lem2} we have
\begin{equation}\label{a41eq}
\frac{1}{4t^2}\Biggl(1-\frac{u^2+u+\frac{3}{2}}{t}\Biggr)\leq\frac{1}{2(t+u+1)}\frac{1}{2t}\Biggl(1-\frac{u^2+\frac{1}{2}}{t}\Biggr)\leq A_{4,1}(t,u) \leq \frac{1}{2(t+u+1)}\frac{1}{2t} \leq \frac{1}{4t^2}.
\end{equation}
Combining \eqref{a41eq} with \eqref{S4lemprfeqn4}, we obtain
\begin{equation*}
\frac{1}{4t^2}\sum_{u=0}^t\frac{\alpha^{2u}}{(2u)!}-\frac{1}{4t^3}\sum_{u=0}^t\frac{(u^2+u+\frac{3}{2})\alpha^{2u}}{(2u)!}\leq s_{4,1}(t) \leq \frac{1}{4t^2}\sum_{u=0}^t\frac{\alpha^{2u}}{(2u)!},
\end{equation*}
and consequently, we get
\begin{equation}\label{S4lemprfeqn5}
\frac{1}{4t^2}\sum_{u=0}^{\infty}\frac{\alpha^{2u}}{(2u)!}-\frac{1}{4t^2}\sum_{u=t+1}^{\infty}\frac{\alpha^{2u}}{(2u)!}-\frac{1}{4t^3}\sum_{u=0}^{\infty}\frac{(u^2+u+\frac{3}{2})\alpha^{2u}}{(2u)!}\leq s_{4,1}(t) \leq \frac{1}{4t^2}\sum_{u=0}^{\infty}\frac{\alpha^{2u}}{(2u)!}.
\end{equation}
Equation \eqref{S4lemprfeqn5} together with Lemmas \ref{helperC}-\ref{helperC2} imply
\begin{equation}\label{s41final}
\begin{split}
\frac{1}{4t^2}\cosh(\alpha)-\frac{3}{5t^3}\leq \frac{1}{4t^2}\cosh(\alpha)-\frac{1}{t^3}\Biggl(\frac{\alpha^4}{72}+\frac{(\alpha^2+6)\cosh(\alpha)}{16}+\frac{3\alpha\sinh(\alpha)}{16}\Biggr)&\leq s_{4,1}(t) \leq\\ &\frac{1}{4t^2}\cosh(\alpha).
\end{split}
\end{equation}
Next, by Lemma \ref{lem3}, we obtain
\begin{equation}\label{a42ineq2}
0\leq A_{4,2}(t,u)-\frac{1}{1+2u}\Biggl(\frac{(-1)^t}{\binom{-\frac{3}{2}}{t}}-\frac{2u+1}{2t}\Biggr)\leq \frac{1}{1+2u}\frac{p_3(u)}{12t^2},
\end{equation}
where $p_3(u)$ is as in \eqref{s22interm}.
Plugging \eqref{a42ineq2} into \eqref{S4lemprfeqn4}, it follows that 
\begin{equation*}
\begin{split}
0 \leq s_{4,2}(t)-\sum_{u=0}^{\infty}\frac{\alpha^{2u}}{(2u+1)!}\Biggl(\frac{(-1)^t}{\binom{-\frac{3}{2}}{t}}-\frac{2u+1}{2t}\Biggr)+&\sum_{u=t+1}^{\infty}\frac{\alpha^{2u}}{(2u+1)!}\Biggl(\frac{(-1)^t}{\binom{-\frac{3}{2}}{t}}-\frac{2u+1}{2t}\Biggr)\\
&\frac{1}{12t^2}\sum_{u=0}^{\infty}\frac{p_3(u)\alpha^{2u}}{(2u+1)!},
\end{split}
\end{equation*}
which implies that
\begin{equation}\label{last42}
\begin{split}
-\frac{(-1)^t}{\binom{-\frac{3}{2}}{t}}\sum_{u=t+1}^{\infty}\frac{\alpha^{2u}}{(2u+1)!}\leq& s_{4,2}(t)-\sum_{u=0}^{\infty}\frac{\alpha^{2u}}{(2u+1)!}\Biggl(\frac{(-1)^t}{\binom{-\frac{3}{2}}{t}}-\frac{2u+1}{2t}\Biggr) \leq\\
& \frac{1}{12t^2}\sum_{u=0}^{\infty}\frac{p_3(u)\alpha^{2u}}{(2u+1)!}+\frac{1}{2t}\sum_{u=t+1}^{\infty}\frac{(2u+1)\alpha^{2u}}{(2u+1)!}.
\end{split}
\end{equation}
By Lemma \ref{helperC}, 
\begin{equation}\label{S4lemprfeqn6}
\sum_{u=t+1}^{\infty}\frac{\alpha^{2u}}{(2u+1)!}\leq \sum_{u=t+1}^{\infty}\frac{\alpha^{2u}}{(2u)!}\leq \frac{\alpha^4}{3\cdot 3!t^2}=\frac{\alpha^4}{18 t^2},
\end{equation}
and
\begin{equation}\label{S4lemprfeqn7}
\sum_{u=t+1}^{\infty}\frac{(2u+1)\alpha^{2u}}{(2u+1)!}\leq \sum_{u=t+1}^{\infty}\frac{(2u+1)\alpha^{2u}}{(2u)!}\leq \frac{C_0+2C_1}{t^2}=\frac{\alpha^4(1+4)}{3\cdot 3!t^2}=\frac{5\alpha^4}{18t^2}.
\end{equation}
Applying \eqref{S4lemprfeqn6} and \eqref{S4lemprfeqn7} to \eqref{last42} and using Lemma \ref{helperC2}, we finally obtain
\begin{equation}\label{s42finalineq}
\begin{split}
-\frac{3}{1000t^2}< -\frac{2}{3}\cdot\frac{\alpha^4}{18}\leq -\frac{(-1)^t}{\binom{-\frac{3}{2}}{t}}\frac{\alpha^4}{18t^2}&\leq s_{4,2}(t)-\frac{(-1)^t}{\binom{-\frac{3}{2}}{t}}\frac{\sinh(\alpha)}{\alpha}+\frac{1}{2t}\cosh(\alpha)\leq \\
& \frac{1}{t^2}\Biggl(\frac{(\alpha^2+6)\cosh(\alpha)}{24}+\frac{\alpha\sinh(\alpha)}{8}+\frac{5\alpha^4}{36}\Biggr)< \frac{7}{20t^2}.
\end{split}
\end{equation}
From \eqref{s41final}, \eqref{s42finalineq}, and \eqref{S4lemprfeqn3}, it follows that
\begin{equation*}
\begin{split}
\frac{1}{t^2}\Biggl(-\frac{3}{1000}+\frac{1}{4}\cosh(\alpha)-\frac{3}{5}\Biggr)&\leq -\frac{3}{1000t^2}-\frac{3}{5t^3}+\frac{1}{4t^2}\cosh(\alpha)\leq\\
&\frac{S_4(t)}{(-1)^t\binom{-\frac{3}{2}}{t}}-\frac{(-1)^t}{\binom{-\frac{3}{2}}{t}}\frac{\sinh(\alpha)}{\alpha}+\frac{1}{2t}\cosh(\alpha)<\frac{1}{t^2}\Bigl(\frac{7}{20}+\frac{1}{4}\cosh(\alpha)\Bigr).
\end{split}
\end{equation*}
This implies for $t \geq 1$,
\begin{equation}\label{sum4final}
\frac{13}{20t^2}> \frac{S_4(t)}{(-1)^t\binom{-\frac{3}{2}}{t}}-\frac{(-1)^t}{\binom{-\frac{3}{2}}{t}}\frac{\sinh(\alpha)}{\alpha}+\frac{1}{2t}\cosh(\alpha)> -\frac{1}{3t^2}.
\end{equation}
\qed

 \section{Error bounds}\label{sec7}
 \begin{lemma}\label{errorlem1}
 For all $n,k \in \mathbb{Z}_{\geq1}$,
 \begin{equation}\label{errorlem1eqn1}
 \dfrac{1}{(24n)^k} < \sum_{t=k}^{\infty}\dfrac{1}{(24n)^t}\leq \dfrac{24}{23} \dfrac{1}{(24n)^k}.
 \end{equation}
 \end{lemma}
\begin{proof}
  The statement follows from
  $$\sum_{t=k}^{\infty}\dfrac{1}{(24n)^t}=\dfrac{1}{(24n)^k}\dfrac{24n}{24n-1}\ \ \text{and}\ \ 1< \dfrac{24n}{24n-1} \leq  \dfrac{24}{23}\ \ \text{for all}\ n \geq 1.$$
\end{proof}

\begin{lemma}\label{errorlem2}
For all $n,k,s \in \mathbb{Z}_{\geq1}$,
\begin{equation}\label{errorlem2eqn1}
\dfrac{1}{(k+1)^{s-\frac{1}{2}}}\dfrac{1}{(24n)^k}<\sum_{t=k}^{\infty} \dfrac{(-1)^t \binom{-\frac{3}{2}}{t}}{t^s}\dfrac{1}{(24n)^t}<\dfrac{12}{5(k+1)^{s-\frac{1}{2}}}\dfrac{1}{(24n)^k}.
\end{equation}	
\end{lemma}
\begin{proof}
Rewrite the infinite sum as
\begin{equation}\label{errorlem2eqn2}
\sum_{t=k}^{\infty} \dfrac{(-1)^t \binom{-\frac{3}{2}}{t}}{t^s}\dfrac{1}{(24n)^t} = \sum_{t=k}^{\infty} \dfrac{\binom{2t+2}{t+1}}{4^t}\dfrac{t+1}{2t^s}\dfrac{1}{(24n)^t}.
\end{equation}
For all $t \geq 1$, 
$$\dfrac{4^t}{2\sqrt{t}}\leq \binom{2t}{t} \leq \dfrac{4^t}{\sqrt{\pi t}}.$$
From \eqref{errorlem2eqn2} we get
\begin{equation}\label{errorlem2eqn3}
\sum_{t=k}^{\infty}\dfrac{\sqrt{t+1}}{t^s} \dfrac{1}{(24n)^t}\leq \sum_{t=k}^{\infty} \dfrac{(-1)^t \binom{-\frac{3}{2}}{t}}{t^s}\dfrac{1}{(24n)^t} \leq \dfrac{4}{\sqrt{\pi}}\sum_{t=k}^{\infty}\dfrac{\sqrt{t+1}}{2t^s} \dfrac{1}{(24n)^t}.
\end{equation}
For all $k \geq 1$,
\begin{equation}\label{errorlem2eqn4}
\sum_{t=k}^{\infty} \dfrac{(-1)^t \binom{-\frac{3}{2}}{t}}{t^s}\dfrac{1}{(24n)^t} \geq \sum_{t=k}^{\infty}\dfrac{\sqrt{t+1}}{t^s} \dfrac{1}{(24n)^t} > \sum_{t=k}^{\infty}\dfrac{1}{(t+1)^{s-\frac{1}{2}}} \dfrac{1}{(24n)^t}> \dfrac{1}{(k+1)^{s-\frac{1}{2}}} \dfrac{1}{(24n)^k}
\end{equation}
and
\begin{eqnarray}\label{errorlem2eqn5}
\sum_{t=k}^{\infty} \dfrac{(-1)^t \binom{-\frac{3}{2}}{t}}{t^s}\dfrac{1}{(24n)^t} \leq \dfrac{4}{\sqrt{\pi}}\sum_{t=k}^{\infty}\dfrac{\sqrt{t+1}}{2t^s} \dfrac{1}{(24n)^t} &<& \dfrac{4}{\sqrt{\pi}}\sum_{t=k}^{\infty}\dfrac{1}{(t+1)^{s-\frac{1}{2}}} \dfrac{1}{(24n)^t}\nonumber\\
 &\leq& \dfrac{4}{\sqrt{\pi}(k+1)^{s-\frac{1}{2}}}\sum_{t=k}^{\infty} \dfrac{1}{(24n)^t}\nonumber\\
 &<& \dfrac{4\cdot 24}{23\cdot\sqrt{\pi}}\dfrac{1}{(k+1)^{s-\frac{1}{2}}} \dfrac{1}{(24n)^k} \ (\text{by}\ \eqref{errorlem1eqn1}).\nonumber\\
 &<& \frac{12}{5}\dfrac{1}{(k+1)^{s-\frac{1}{2}}} \dfrac{1}{(24n)^k}.
\end{eqnarray}
Equations \eqref{errorlem2eqn4} and \eqref{errorlem2eqn5} imply \eqref{errorlem2eqn1}.
\end{proof}

\begin{lemma}\label{errorlem3}
	For $n \in \mathbb{Z}_{\geq 1}$and $k \in \mathbb{Z}_{\geq0}$,
	\begin{equation}\label{errorlem3eqn1}
	0<\sum_{t=k}^{\infty} \binom{-\frac{3}{2}}{t}\dfrac{(-1)^t}{(24n)^t}<4\sqrt{2}\dfrac{\sqrt{k+1}}{(24n)^k}.
	\end{equation}	
\end{lemma}
\begin{proof}
Setting $(n,s) \mapsto (24n,2)$ in \eqref{BPRZeqn1}, it follows that for all $n \geq 1$,
\begin{equation*}
0< \sum_{t=k}^{\infty} \binom{-\frac{3}{2}}{t}\dfrac{(-1)^t}{(24n)^t} < 4\sqrt{2}\dfrac{\sqrt{k+1}}{(24n)^k}.
\end{equation*}
\end{proof}

\begin{definition}\label{newdef11}
For all $k \geq 1$ define
\begin{equation*}
\begin{split}
L_1(k)&:= \Biggl(\cosh(\alpha)-\dfrac{6\alpha\sinh(\alpha)}{5\sqrt{k+1}}-\dfrac{3}{10 (k+1)^{3/2}}\Biggr)\Bigl(\dfrac{1}{\sqrt{24}}\Bigr)^{2k}\\
\and  \hspace{2 cm}&\\
U_1(k)&:= \Biggl(\dfrac{24\cosh(\alpha)}{23}-\dfrac{\alpha\sinh(\alpha)}{2\sqrt{k+1}}+\dfrac{5}{4 (k+1)^{3/2}}\Biggr)\Bigl(\dfrac{1}{\sqrt{24}}\Bigr)^{2k}.
\end{split}
\end{equation*}
\end{definition}

\begin{lemma}\label{errorsum1}
Let $L_1(k)$ and $U_1(k)$ be as in Definition \ref{newdef11}. Let $g_{e,1}(t)$ be as in Definition \ref{newdef6}. Then for all $n,k \in \mathbb{Z}_{\geq 1}$, 
\begin{equation}\label{errorsum1eqn1}
L_1(k)\Bigl(\dfrac{1}{\sqrt{n}}\Bigr)^{2k}<\sum_{t=k}^{\infty}g_{e,1}(t) \Bigl(\dfrac{1}{\sqrt{n}}\Bigr)^{2t} <U_1(k)\Bigl(\dfrac{1}{\sqrt{n}}\Bigr)^{2k}.
\end{equation}
\end{lemma}
\begin{proof}
From \eqref{eqn28} and \eqref{S1lemmaeqn}, it follows that for $t \geq 1$,
\begin{equation}\label{errorsum1eqn2}
\begin{split}
\cosh(\alpha) -\dfrac{(-1)^t \binom{-\frac{3}{2}}{t}}{2t}\alpha\sinh(\alpha)-\dfrac{1}{8}\dfrac{(-1)^t \binom{-\frac{3}{2}}{t}}{t^2}&<(24)^t g_{e,1}(t)=1+S_1(t)\\
 < \cosh(\alpha) &-\dfrac{(-1)^t \binom{-\frac{3}{2}}{t}}{2t}\alpha\sinh(\alpha)+\dfrac{13}{25}\dfrac{(-1)^t \binom{-\frac{3}{2}}{t}}{t^2}.
\end{split}
\end{equation}
Now, applying \eqref{errorlem1eqn1} and \eqref{errorlem2eqn1} with $s=1\ \text{and}\ 2$, respectively, to \eqref{errorsum1eqn2}, it follows that for all $k \geq 1$,
$$\sum_{t=k}^{\infty}g_{e,1}(t)\Bigl(\dfrac{1}{\sqrt{n}}\Bigr)^{2t} > \Biggl(\cosh(\alpha)-\dfrac{6\alpha\sinh(\alpha)}{5\sqrt{k+1}}-\dfrac{3}{10 (k+1)^{3/2}}\Biggr) \Bigl(\dfrac{1}{\sqrt{24n}}\Bigr)^{2k} $$
and 
\begin{equation*}
\begin{split}
\sum_{t=k}^{\infty}g_{e,1}(t)\Bigl(\dfrac{1}{\sqrt{n}}\Bigr)^{2t}&<\Biggl(\dfrac{24\cosh(\alpha)}{23}-\dfrac{\alpha\sinh(\alpha)}{2\sqrt{k+1}}+\dfrac{13\cdot 12}{25 \cdot 5}\dfrac{1}{ (k+1)^{3/2}}\Biggr) \Bigl(\dfrac{1}{\sqrt{24n}}\Bigr)^{2k}\\
&<\Biggl(\dfrac{24\cosh(\alpha)}{23}-\dfrac{\alpha\sinh(\alpha)}{2\sqrt{k+1}}+\frac{5}{4}\dfrac{1}{ (k+1)^{3/2}}\Biggr) \Bigl(\dfrac{1}{\sqrt{24n}}\Bigr)^{2k}.
\end{split}
\end{equation*}
\end{proof}

\begin{definition}\label{newdef12}
	For all $k \geq 1$, define
	\begin{equation*}
\begin{split}
L_2(k)&:= \Biggl(-\dfrac{24\cosh(\alpha)}{23}-\dfrac{12}{5\sqrt{k+1}}\Biggr)\Bigl(\dfrac{1}{\sqrt{24}}\Bigr)^{2k}\\
\and  \hspace{2 cm}&\\
U_2(k)&:= \Biggl(-\cosh(\alpha)+\dfrac{4\sqrt{2}\sinh(\alpha)}{\alpha}\sqrt{k+1}+\dfrac{66}{25 \sqrt{k+1}}\Biggr)\Bigl(\dfrac{1}{\sqrt{24}}\Bigr)^{2k}.
\end{split}
\end{equation*}
\end{definition}

\begin{lemma}\label{errorsum2}
Let $L_2(k)$ and $U_2(k)$ be as in Definition \ref{newdef12}. Let $g_{e,2}(t)$ be as in Definition \ref{newdef7}. Then for all $n,k \in \mathbb{Z}_{\geq 1}$,
	\begin{equation}\label{errorsum2eqn1}
	L_2(k)\Bigl(\dfrac{1}{\sqrt{n}}\Bigr)^{2k}<\sum_{t=k}^{\infty}g_{e,2}(t) \Bigl(\dfrac{1}{\sqrt{n}}\Bigr)^{2t} <U_2(k)\Bigl(\dfrac{1}{\sqrt{n}}\Bigr)^{2k}.
	\end{equation}
\end{lemma}
\begin{proof}
	From \eqref{eqn32} and \eqref{S2lemmaeqn}, it follows that for $t \geq 1$,
	\begin{equation}\label{errorsum2eqn2}
	\begin{split}
	-\cosh(\alpha) +(-1)^t \binom{-\frac{3}{2}}{t}\dfrac{\sinh(\alpha)}{\alpha}-\dfrac{(-1)^t \binom{-\frac{3}{2}}{t}}{t}&<(24)^t g_{e,2}(t)=(-1)^{t-1}S_2(t)\\
	< -\cosh(\alpha) &+(-1)^t \binom{-\frac{3}{2}}{t}\dfrac{\sinh(\alpha)}{\alpha}+\dfrac{11}{10}\dfrac{(-1)^t \binom{-\frac{3}{2}}{t}}{t}.
	\end{split}
	\end{equation}
	Now, applying \eqref{errorlem1eqn1}, \eqref{errorlem2eqn1} with $s=1$ and \eqref{errorlem3eqn1} to \eqref{errorsum2eqn2}, it follows that for all $k \geq 1$,
	$$\sum_{t=k}^{\infty}g_{e,2}(t)\Bigl(\dfrac{1}{\sqrt{n}}\Bigr)^{2t} > \Biggl(-\dfrac{24\cosh(\alpha)}{23}-\dfrac{12}{5\sqrt{k+1}}\Biggr) \Bigl(\dfrac{1}{\sqrt{24n}}\Bigr)^{2k} $$
	and 
	$$ \sum_{t=k}^{\infty}g_{e,2}(t)\Bigl(\dfrac{1}{\sqrt{n}}\Bigr)^{2t} < \Biggl(-\cosh(\alpha)+\dfrac{4\sqrt{2}\sinh(\alpha)}{\alpha}\sqrt{k+1}+\dfrac{66}{25}\dfrac{1}{ \sqrt{k+1}}\Biggr) \Bigl(\dfrac{1}{\sqrt{24n}}\Bigr)^{2k}.$$
\end{proof}

\begin{definition}\label{newdef13}
For all $k \geq 1$, define
	\begin{equation*}
\begin{split}
L_3(k)&:= \Biggl(\dfrac{19}{10}\alpha\sinh(\alpha)-\dfrac{109}{10}\cosh(\alpha)\sqrt{k+1}-\dfrac{23}{10}\dfrac{1}{\sqrt{k+1}}\Biggr)\Bigl(\dfrac{1}{\sqrt{24}}\Bigr)^{2k+1}\\
\and  \hspace{2 cm}&\\
U_3(k)&:= \Biggl(2\alpha\sinh(\alpha)+\dfrac{33}{10}\dfrac{1}{\sqrt{k+1}}\Biggr)\Bigl(\dfrac{1}{\sqrt{24}}\Bigr)^{2k+1}.
\end{split}
\end{equation*}
\end{definition}

\begin{lemma}\label{errorsum3}
Let $L_3(k)$ and $U_3(k)$ be as in Definition \ref{newdef13}. Let $g_{o,1}(t)$ be as in Definition \ref{newdef8}. Then for all $n,k \in \mathbb{Z}_{\geq 1}$,
	\begin{equation}\label{errorsum3eqn1}
	L_3(k)\Bigl(\dfrac{1}{\sqrt{n}}\Bigr)^{2k+1}<\sum_{t=k}^{\infty}g_{o,1}(t) \Bigl(\dfrac{1}{\sqrt{n}}\Bigr)^{2t+1} <U_3(k)\Bigl(\dfrac{1}{\sqrt{n}}\Bigr)^{2k+1}.
	\end{equation}
\end{lemma}
\begin{proof}
Define $c_{1}(t):=-\dfrac{6}{\pi}(-1)^t \binom{-\frac{3}{2}}{t}$.
	From \eqref{eqn36} and \eqref{S3lemmaeqn}, it follows that for $t \geq 2$,
	\begin{equation}\label{errorsum3eqn2}
	\begin{split}
	\dfrac{6}{\pi}\alpha\sinh(\alpha)-\dfrac{6}{\pi}\cosh(\alpha)(-1)^t \binom{-\frac{3}{2}}{t}-\dfrac{12\cdot 6}{25\cdot \pi}\dfrac{(-1)^t \binom{-\frac{3}{2}}{t}}{t}&<(\sqrt{24})^{2t+1} g_{o,1}(t)=c_{1}(t)\Biggl(1+\dfrac{S_3(t)}{\binom{-\frac{3}{2}}{t}}\Biggr)\\
	< 	\dfrac{6}{\pi}\alpha\sinh(\alpha) -\dfrac{6}{\pi}\cosh(\alpha)(-1)^t \binom{-\frac{3}{2}}{t}&+\dfrac{71\cdot 6}{100\cdot \pi}\dfrac{(-1)^t \binom{-\frac{3}{2}}{t}}{t}.
	\end{split}
	\end{equation}
A numerical check confirms that \eqref{errorsum3eqn2} also holds for $t=1$; see \eqref{eqn36}.  Now, applying \eqref{errorlem1eqn1}, \eqref{errorlem2eqn1} with $s=1$, and \eqref{errorlem3eqn1} to \eqref{errorsum3eqn2}, it follows that for all $k \geq 1$,
	\begin{equation*}
	\begin{split}
	\sum_{t=k}^{\infty}g_{o,1}(t)\Bigl(\dfrac{1}{\sqrt{n}}\Bigr)^{2t+1} &> \Biggl(\dfrac{6}{\pi}\alpha\sinh(\alpha)-\dfrac{6\cdot 4\sqrt{2}}{\pi}\cosh(\alpha)\sqrt{k+1}-\dfrac{12 \cdot 6\cdot 12}{25\cdot 5\cdot \pi}\dfrac{1}{\sqrt{k+1}}\Biggr) \Bigl(\dfrac{1}{\sqrt{24n}}\Bigr)^{2k+1}\\
	&>\Biggl(\dfrac{19}{10}\alpha\sinh(\alpha)-\dfrac{109}{10}\cosh(\alpha)\sqrt{k+1}-\dfrac{23}{10}\dfrac{1}{\sqrt{k+1}}\Biggr) \Bigl(\dfrac{1}{\sqrt{24n}}\Bigr)^{2k+1}
	\end{split}
	\end{equation*}
	and 
	\begin{equation*}
	\begin{split}
	\sum_{t=k}^{\infty}g_{o,1}(t)\Bigl(\dfrac{1}{\sqrt{n}}\Bigr)^{2t+1} &< \Biggl(\dfrac{6\cdot 24}{23\cdot\pi}\alpha\sinh(\alpha)+\dfrac{71\cdot 6 \cdot 12}{100 \cdot 5\cdot \pi}\dfrac{1}{\sqrt{k+1}}\Biggr) \Bigl(\dfrac{1}{\sqrt{24n}}\Bigr)^{2k+1} \\
	&< \Biggl(2\alpha\sinh(\alpha)+\dfrac{33}{10}\dfrac{1}{\sqrt{k+1}}\Biggr) \Bigl(\dfrac{1}{\sqrt{24n}}\Bigr)^{2k+1}.
	\end{split}
	\end{equation*}
\end{proof}

\begin{definition}\label{newdef14}
	For all $k \geq 1$, define
	\begin{equation*}
\begin{split}
L_4(k)&:= \Biggl(\dfrac{1}{4}\dfrac{\cosh(\alpha)}{\sqrt{k+1}}-\dfrac{11}{20}\alpha\sinh(\alpha)-\dfrac{41}{50}\dfrac{1}{(k+1)^{3/2}}\Biggr)\Bigl(\dfrac{1}{\sqrt{24}}\Bigr)^{2k+1}\\
\and  \hspace{2 cm}&\\
U_4(k)&:= \Biggl(\dfrac{63}{100}\dfrac{\cosh(\alpha)}{\sqrt{k+1}}-\dfrac{13}{25}\alpha\sinh(\alpha)+\dfrac{21}{50}\dfrac{1}{(k+1)^{3/2}}\Biggr)\Bigl(\dfrac{1}{\sqrt{24}}\Bigr)^{2k+1}.
\end{split}
\end{equation*}
\end{definition}

\begin{lemma}\label{errorsum4}
Let $L_4(k)$ and $U_4(k)$ be as in Definition \ref{newdef14}. Let $g_{o,2}$ be as in Definition \ref{newdef9}. Then for all $n,k \in \mathbb{Z}_{\geq 1}$, 
	\begin{equation}\label{errorsum4eqn1}
	L_4(k)\Bigl(\dfrac{1}{\sqrt{n}}\Bigr)^{2k+1}<\sum_{t=k}^{\infty}g_{o,2}(t) \Bigl(\dfrac{1}{\sqrt{n}}\Bigr)^{2t+1} <U_4(k)\Bigl(\dfrac{1}{\sqrt{n}}\Bigr)^{2k+1}.
	\end{equation}
\end{lemma}
\begin{proof}
	Define $c_{2}(t):=-\dfrac{\pi}{6}(-1)^t \binom{-\frac{3}{2}}{t}$.
	From \eqref{eqn40} and \eqref{S4lemmaeqn}, it follows that for $t \geq 1$,
	\begin{equation}\label{errorsum4eqn2}
	\begin{split}
	\dfrac{\pi}{6\cdot 2}\cosh(\alpha)\dfrac{(-1)^t \binom{-\frac{3}{2}}{t}}{t}-\dfrac{\pi}{6}\alpha\sinh(\alpha)-\dfrac{13\cdot \pi}{20\cdot 6}\dfrac{(-1)^t \binom{-\frac{3}{2}}{t}}{t^2}&<(\sqrt{24})^{2t+1} g_{o,2}(t)=c_{2}(t)\dfrac{S_4(t)}{(-1)^t\binom{-\frac{3}{2}}{t}}\\
	< 	\dfrac{\pi}{6\cdot 2}\cosh(\alpha)\dfrac{(-1)^t \binom{-\frac{3}{2}}{t}}{t}-\dfrac{\pi}{6}\alpha\sinh(\alpha)&+\dfrac{\pi}{6\cdot 3}\dfrac{(-1)^t \binom{-\frac{3}{2}}{t}}{t^2}.
	\end{split}
	\end{equation}
	Now, applying \eqref{errorlem1eqn1} and \eqref{errorlem2eqn1} with $s=1 \ \text{and}\ 2$, respectively, to \eqref{errorsum3eqn2}, it follows that for all $k \geq 1$,
	\begin{equation*}
	\begin{split}
	\sum_{t=k}^{\infty}g_{o,2}(t)\Bigl(\dfrac{1}{\sqrt{n}}\Bigr)^{2t+1} &> \Biggl(\dfrac{\pi}{12}\dfrac{\cosh(\alpha)}{\sqrt{k+1}}-\dfrac{24 \cdot \pi}{23 \cdot 6}\alpha\sinh(\alpha)-\dfrac{13\cdot 12\cdot \pi}{20 \cdot 6 \cdot 5}\dfrac{1}{(k+1)^{3/2}}\Biggr) \Bigl(\dfrac{1}{\sqrt{24n}}\Bigr)^{2k+1}\\
	&>\Biggl(\dfrac{1}{4}\dfrac{\cosh(\alpha)}{\sqrt{k+1}}-\dfrac{11}{20}\alpha\sinh(\alpha)-\dfrac{41}{50}\dfrac{1}{(k+1)^{3/2}}\Biggr) \Bigl(\dfrac{1}{\sqrt{24n}}\Bigr)^{2k+1}
	\end{split}
	\end{equation*}
	and 
	\begin{equation*}
	\begin{split}
	\sum_{t=k}^{\infty}g_{o,2}(t)\Bigl(\dfrac{1}{\sqrt{n}}\Bigr)^{2t+1} &< \Biggl(\dfrac{12 \cdot \pi}{12 \cdot 5}\dfrac{\cosh(\alpha)}{\sqrt{k+1}}-\dfrac{\pi}{6}\alpha\sinh(\alpha)+\dfrac{12\cdot \pi}{6 \cdot 3\cdot 5}\dfrac{1}{(k+1)^{3/2}}\Biggr) \Bigl(\dfrac{1}{\sqrt{24n}}\Bigr)^{2k+1} \\
	&< \Biggl(\dfrac{63}{100}\dfrac{\cosh(\alpha)}{\sqrt{k+1}}-\dfrac{13}{25}\alpha\sinh(\alpha)+\dfrac{21}{50}\dfrac{1}{(k+1)^{3/2}}\Biggr) \Bigl(\dfrac{1}{\sqrt{24n}}\Bigr)^{2k+1}.
	\end{split}
	\end{equation*}
\end{proof}
\begin{definition}\label{newdef15}
	For $k \geq 1$, define
	\begin{equation*}
	\widehat{L}_2(k):=\dfrac{1}{\alpha^k}\dfrac{1}{\sqrt{24}^k}\Bigl(1-\dfrac{1}{4\sqrt{n}}\Bigr)\ \ \and\ \ \widehat{U}_2(k):=\dfrac{1}{\alpha^k}\dfrac{1}{\sqrt{24}^k}\Bigl(1+\dfrac{k}{3n}\Bigr).
	\end{equation*}
\end{definition}
\begin{definition}\label{newdef16}
For $k \geq 1$, define
$$n_0(k):=\dfrac{k+2}{24}.$$	
\end{definition}
\begin{lemma}\label{errorlem5} Let $\widehat{L}_2(k)$, and $\widehat{U}_2(k)$ be as in Definition \ref{newdef15}. Let $n_0(k)$ be as in Definition \ref{newdef16}. Then for all $k \in \mathbb{Z}_{\geq 1}$ and $n > n_0(k)$,
	\begin{equation}\label{errorlem5eqn1}
	\frac{e^{\pi\sqrt{2n/3}}}{4n\sqrt{3}}\frac{\widehat{L}_2(k)}{\sqrt{n}^k}<\dfrac{\sqrt{12}\ e^{\mu(n)}}{24n-1}\frac{1}{\mu(n)^k}<\frac{e^{\pi\sqrt{2n/3}}}{4n\sqrt{3}}\frac{\widehat{U}_2(k)}{\sqrt{n}^k}.
	\end{equation}	
\end{lemma}
\begin{proof}
Define 
\begin{equation*}
\mathcal{E}(n,k):=\dfrac{\sqrt{12}\ e^{\mu(n)}}{24n-1}\frac{1}{\mu(n)^k},\  \mathcal{U}(n,k)= \frac{e^{\pi\sqrt{2n/3}}}{4n\sqrt{3}}\frac{1}{\sqrt{n}^k}
\end{equation*}
and
$$\mathcal{Q}(n,k):=\dfrac{\mathcal{E}(n,k)}{\mathcal{U}(n,k)}=\mathcal{Q}(n,k)=\dfrac{e^{\pi\sqrt{\frac{2n}{3}}\Bigl(\sqrt{1-\frac{1}{24n}}-1\Bigr)}}{\alpha^k}\dfrac{1}{\sqrt{24}^k}\Bigl(1-\dfrac{1}{24n}\Bigr)^{-\frac{k+2}{2}}.$$
Using \eqref{BPRZ1eqn1} with $(m,n,s) \mapsto (1,24n,1)$, we obtain for all $n \geq 1$,
\begin{equation*}
-\dfrac{1}{12n} <\sqrt{1-\dfrac{1}{24n}}-1=\sum_{m=1}^{\infty}\binom{1/2}{m}\dfrac{(-1)^m}{(24n)^m}<0,
\end{equation*}
which implies that for $n \geq 1$, 
\begin{equation}\label{errorlem5eqn2}
\Bigl(1-\dfrac{1}{4\sqrt{n}}\Bigr)< e^{-\frac{\pi}{12}\sqrt{\frac{2}{3n}}}<e^{\pi\sqrt{\frac{2n}{3}}\Bigl(\sqrt{1-\frac{1}{24n}}-1\Bigr)}<1.
\end{equation}
Hence
\begin{equation}\label{errorlem5eqn3}
\dfrac{1}{(\alpha\cdot \sqrt{24})^k}\Bigl(1-\dfrac{1}{24n}\Bigr)^{-\frac{k+2}{2}}\Bigl(1-\dfrac{1}{4\sqrt{n}}\Bigr)<\mathcal{Q}(n,k)<\dfrac{1}{(\alpha\cdot \sqrt{24})^k}\Bigl(1-\dfrac{1}{24n}\Bigr)^{-\frac{k+2}{2}}.
\end{equation}
In order to estimate $\Bigl(1-\dfrac{1}{24n}\Bigr)^{-\frac{k+2}{2}}$, we need to split into two cases depending on $k$ is even or odd.

For $k=2\ell$ with $\ell \in \mathbb{Z}_{\geq 0}$: 
\begin{equation*}
\Bigl(1-\dfrac{1}{24n}\Bigr)^{-\frac{k+2}{2}}=\Bigl(1-\dfrac{1}{24n}\Bigr)^{-(\ell+1)} =1+\sum_{j=1}^{\infty}\binom{-(\ell+1)}{j}\dfrac{(-1)^j}{(24n)^j}.
\end{equation*}
From \eqref{BPRZ2eqn1} with $(m,s,n)\mapsto (1,\ell+1,24n)$, for all $n > \frac{\ell+1}{12}$, we get
\begin{equation*}
0< \sum_{j=1}^{\infty}\binom{-(\ell+1)}{j}\dfrac{(-1)^j}{(24n)^j} < \beta_{1,24n}(\ell+1)=\dfrac{\ell+1}{12n},
\end{equation*}
which is equivalent to
\begin{equation}\label{errorlem5eqn4}
1< \Bigl(1-\dfrac{1}{24n}\Bigr)^{-\frac{k+2}{2}}<1+\dfrac{k+2}{24n}\ \ \text{for all}\ \ n > \frac{k+2}{24}.
\end{equation}

For $k=2\ell+1$ with $\ell \in \mathbb{Z}_{\geq 0}$: 
\begin{equation*}
\Bigl(1-\dfrac{1}{24n}\Bigr)^{-\frac{k+2}{2}}=\Bigl(1-\dfrac{1}{24n}\Bigr)^{-\frac{2\ell+3}{2}} =1+\sum_{j=1}^{\infty}\binom{-\frac{2\ell+3}{2}}{j}\dfrac{(-1)^j}{(24n)^j}.
\end{equation*}
Using \eqref{BPRZeqn1} with $(m,s,n)\mapsto (1,\ell+2,24n)$, for all $n > \frac{\ell+2}{24}$, we get
\begin{equation*}
0< \sum_{j=1}^{\infty}\binom{-\frac{2\ell+3}{2}}{j}\dfrac{(-1)^j}{(24n)^j} < b_{1,24n}(\ell+2)=\dfrac{\ell+2}{6n}
\end{equation*}
which is equivalent to
\begin{equation}\label{errorlem5eqn5}
1< \Bigl(1-\dfrac{1}{24n}\Bigr)^{-\frac{k+2}{2}}<1+\dfrac{k+3}{12n} \leq 1+\dfrac{k}{3n}\ \ \text{for all}\ \ n > \frac{k+3}{48}.
\end{equation}
From \eqref{errorlem5eqn4} and \eqref{errorlem5eqn5}, for all $n > \frac{k+2}{24}$ it follows that
\begin{equation}\label{errorlem5eqn6}
1< \Bigl(1-\dfrac{1}{24n}\Bigr)^{-\frac{k+2}{2}}<1+\dfrac{k+3}{12n} \leq 1+\dfrac{k}{3n}.
\end{equation}
Combining \eqref{errorlem5eqn3} and \eqref{errorlem5eqn6} concludes the proof.

\end{proof}

\section{Main Theorem}\label{mainthmsec}
\begin{definition}\label{newdef17}
	For $w \in \mathbb{Z}_{\geq 1}$, define
	\begin{equation*}
	\bigl(\gamma_0(w), \gamma_1(w)\bigr):=
	\begin{cases}
	(23,24), \quad \text{if $w$ is even}\\
	(15,17), \quad \text{if $w$ is odd}
	\end{cases}.
	\end{equation*}
\end{definition}
\begin{definition}\label{newdef18}
Let $\gamma_0(w)$ and $\gamma_1(w)\bigr)$ be as in Definition \ref{newdef17}. Then for all $w \in \mathbb{Z}_{\geq 1}$, define
\begin{equation*}
L(w):=-\gamma_0(w) \dfrac{\sqrt{\lceil w/2\rceil+1}}{\sqrt{24}^w}\ \ \text{and}\ \ U(w):=\gamma_1(w) \dfrac{\sqrt{\lceil w/2\rceil+1}}{\sqrt{24}^w}.
\end{equation*}	
\end{definition}

\begin{lemma}\label{mainlemmaevencase}
Let $\widehat{g}(k)$ be as in Theorem \ref{bprz1} and $n_0(k)$ as in Definition \ref{newdef16}. Let $g(t)$ be as in \eqref{g(t)definition}. Let $L(w)$ and $U(w)$ be as in Definition \ref{newdef18}. If $m \in \mathbb{Z}_{\geq 1}$ and $n>\max\bigl\{1, n_0(2m), \widehat{g}(2m)\bigr\}$, then
\begin{equation*}
\frac{e^{\pi\sqrt{2n/3}}}{4n\sqrt{3}} \Biggl(\sum_{t=0}^{2m-1}g(t)\Bigl(\dfrac{1}{\sqrt{n}}\Bigr)^t+\dfrac{L(2m)}{\sqrt{n}^{2m}}\Biggr)< p(n) < \frac{e^{\pi\sqrt{2n/3}}}{4n\sqrt{3}} \Biggl(\sum_{t=0}^{2m-1}g(t)\Bigl(\dfrac{1}{\sqrt{n}}\Bigr)^t+\dfrac{U(2m)}{\sqrt{n}^{2m}}\Biggr).
\end{equation*}
\end{lemma}

\begin{proof} Recalling Definition \ref{newdef10}, from Lemma \ref{newlemma2}, we have
 	\begin{eqnarray}\label{mainthmseceqn3}
	\sum_{t=0}^{\infty}g(t)\Bigl(\dfrac{1}{\sqrt{n}}\Bigr)^t&=&\sum_{t=0}^{2m-1}g(t)\Bigl(\dfrac{1}{\sqrt{n}}\Bigr)^t+\sum_{t=2m}^{\infty}g(t)\Bigl(\dfrac{1}{\sqrt{n}}\Bigr)^t\nonumber\\
	&=& \sum_{t=0}^{2m-1}g(t)\Bigl(\dfrac{1}{\sqrt{n}}\Bigr)^t+\sum_{t=m}^{\infty}g(2t)\Bigl(\dfrac{1}{\sqrt{n}}\Bigr)^{2t}+\sum_{t=m}^{\infty}g(2t+1)\Bigl(\dfrac{1}{\sqrt{n}}\Bigr)^{2t+1}\nonumber\\
	&=& \sum_{t=0}^{2m-1}g(t)\Bigl(\dfrac{1}{\sqrt{n}}\Bigr)^t+\sum_{t=m}^{\infty}(g_{e,1}(t)+g_{e,2}(t))\Bigl(\dfrac{1}{\sqrt{n}}\Bigr)^{2t}+\sum_{t=m}^{\infty}(g_{o,1}(t)+g_{o,2}(t))\Bigl(\dfrac{1}{\sqrt{n}}\Bigr)^{2t+1}.\nonumber\\
	\end{eqnarray}
	Using Lemmas \ref{errorsum1}-\ref{errorsum4} by assigning $k \mapsto m$, it follows that
	\begin{eqnarray}\label{mainthmseceqn4}
	\dfrac{L_1(m)+L_2(m)}{\sqrt{n}^{2m}}+\dfrac{L_3(m)+L_4(m)}{\sqrt{n}^{2m+1}}&<&\sum_{t=2m}^{\infty}g(t)\Bigl(\dfrac{1}{\sqrt{n}}\Bigr)^t<\dfrac{U_1(m)+U_2(m)}{\sqrt{n}^{2m}}+\dfrac{U_3(m)+U_4(m)}{\sqrt
		{n}^{2m+1}}.\nonumber\\
	\end{eqnarray}	
	Moreover, by Lemma \ref{errorlem5} with $k=2m$, it follows that
	\begin{equation}\label{mainthmseceqn5}
	\dfrac{\sqrt{12}\ e^{\mu(n)}}{24n-1}\frac{1}{\mu(n)^{2m}}<\frac{e^{\pi\sqrt{2n/3}}}{4n\sqrt{3}}\frac{\widehat{U}_2(2m)}{\sqrt{n}^{2m}}.
	\end{equation}
	Finally, from \eqref{mainthmseceqn4} and \eqref{mainthmseceqn5} along with the fact that $U_3(m)+U_4(m)>0$, we obtain
	\begin{equation}\label{finalestim1}
	\begin{split}
	\dfrac{\sqrt{12}e^{\mu(n)}}{24n-1}\Biggl(1-\dfrac{1}{\mu(n)}+\dfrac{1}{\mu(n)^{2m}}\Biggr) &< \frac{e^{\pi\sqrt{2n/3}}}{4n\sqrt{3}} \Biggl(\sum_{t=0}^{2m-1}g(t)\Bigl(\dfrac{1}{\sqrt{n}}\Bigr)^t+\dfrac{\sum_{i=1}^{4}U_i(m)+\widehat{U}_2(2m)}{\sqrt{n}^{2m}}\Biggr).
	\end{split}
	\end{equation}
	Since for all $m \geq 1$, $L_3(m)+L_4(m)<0$, it follows that
	\begin{equation}\label{finalestim2}
	\begin{split}
	\dfrac{\sqrt{12}e^{\mu(n)}}{24n-1}\Biggl(1-\dfrac{1}{\mu(n)}-\dfrac{1}{\mu(n)^{2m}}\Biggr) &> \frac{e^{\pi\sqrt{2n/3}}}{4n\sqrt{3}} \Biggl(\sum_{t=0}^{2m-1}g(t)\Bigl(\dfrac{1}{\sqrt{n}}\Bigr)^t+\dfrac{\sum_{i=1}^{4}L_i(m)-\widehat{U}_2(2m)}{\sqrt{n}^{2m}}\Biggr).
	\end{split}
	\end{equation}
	From Lemmas \ref{errorsum1}-\ref{errorsum4} and \ref{errorlem5}, for all $n \geq \max \{1,n_0(2m)\}$, 
	\begin{eqnarray}\label{evencaseupperbound1}
	\sum_{i=1}^{4}U_i(m)+\widehat{U}_2(2m)
	&<& \Biggl(4+\dfrac{4}{\sqrt{m+1}}+\dfrac{2}{(m+1)^{3/2}}+6\sqrt{m+1}+\dfrac{2m}{3\alpha^2 n}\Biggr)\dfrac{1}{\sqrt{24}^{2m}}\nonumber.
	\end{eqnarray}
	 For all $1\leq m \leq 10$ observe that $n_0(2m)<1$ and therefore, $\dfrac{2m}{3\alpha^2 n}<\dfrac{20}{3\alpha^2}<25$; whereas for $m \geq 11$, $n_0(2m)>1$. Consequently, $\dfrac{2m}{3\alpha^2 n}<\dfrac{8m}{\alpha^2(m+1)}<\dfrac{8}{\alpha^2}<10$; i.e., $\dfrac{2m}{3\alpha^2 n}<25$.\\
	 Continuing our estimation
	\begin{eqnarray}\label{evencaseupperbound}
	\sum_{i=1}^{4}U_i(m)+\widehat{U}_2(2m)
	&<& \Biggl(29+\dfrac{4}{\sqrt{m+1}}+\dfrac{2}{(m+1)^{3/2}}+6\sqrt{m+1}\Biggr)\dfrac{1}{\sqrt{24}^{2m}}\nonumber\\
	& \leq & \dfrac{24 \sqrt{m+1}}{\sqrt{24}^{2m}}=U(2m).
	\end{eqnarray}
	Similarly, for all $n \geq \max \{1,n_0(2m)\}$, 
	\begin{eqnarray}\label{evencaselowerbound}
	\sum_{i=1}^{4}L_i(m)-\widehat{U}_2(2m)&>& \Biggl(-29-\dfrac{4}{\sqrt{m+1}}-\dfrac{1}{2(m+1)^{3/2}}-3\sqrt{m+1}\Biggr)\dfrac{1}{\sqrt{24}^{2m}} \nonumber\\
	&\geq & -\dfrac{23\sqrt{m+1}}{\sqrt{24}^{2m}}=L(2m).
	\end{eqnarray}
Plugging \eqref{evencaseupperbound} and \eqref{evencaselowerbound} into \eqref{finalestim1} and \eqref{finalestim2}, respectively,  and applying Theorem \ref{bprz1}, we get
\begin{equation}\label{finalestimevenupper}
p(n)<\dfrac{\sqrt{12}e^{\mu(n)}}{24n-1}\Biggl(1-\dfrac{1}{\mu(n)}+\dfrac{1}{\mu(n)^{2m}}\Biggr) < \frac{e^{\pi\sqrt{2n/3}}}{4n\sqrt{3}} \Biggl(\sum_{t=0}^{2m-1}g(t)\Bigl(\dfrac{1}{\sqrt{n}}\Bigr)^t+\dfrac{U(2m)}{\sqrt{n}^{2m}}\Biggr)
\end{equation}
and 
\begin{equation}\label{finalestimevenlower}
p(n)>\dfrac{\sqrt{12}e^{\mu(n)}}{24n-1}\Biggl(1-\dfrac{1}{\mu(n)}-\dfrac{1}{\mu(n)^{2m}}\Biggr) > \frac{e^{\pi\sqrt{2n/3}}}{4n\sqrt{3}} \Biggl(\sum_{t=0}^{2m-1}g(t)\Bigl(\dfrac{1}{\sqrt{n}}\Bigr)^t+\dfrac{L(2m)}{\sqrt{n}^{2m}}\Biggr).
\end{equation}
\end{proof}

\begin{lemma}\label{mainlemmaoddcase}
	Let $\widehat{g}(k)$ be as in Theorem \ref{bprz1} and $n_0(k)$ as in Definition \ref{newdef16}. Let $g(t)$ be as in Equation \eqref{g(t)definition}. Let $L(w)$ and $U(w)$ be as in Definition \ref{newdef18}. If $m \in \mathbb{Z}_{\geq 0}$ and $n>\max \bigl\{1, n_0(2m+1), \widehat{g}(2m+1)\bigr\}$, then
	\begin{equation*}
	\frac{e^{\pi\sqrt{2n/3}}}{4n\sqrt{3}} \Biggl(\sum_{t=0}^{2m}g(t)\Bigl(\dfrac{1}{\sqrt{n}}\Bigr)^t+\dfrac{L(2m+1)}{\sqrt{n}^{2m+1}}\Biggr)< p(n) < \frac{e^{\pi\sqrt{2n/3}}}{4n\sqrt{3}} \Biggl(\sum_{t=0}^{2m}g(t)\Bigl(\dfrac{1}{\sqrt{n}}\Bigr)^t+\dfrac{U(2m+1)}{\sqrt{n}^{2m+1}}\Biggr).
	\end{equation*}
\end{lemma}

\begin{proof} Recalling Definition \ref{newdef10}, by Lemma \ref{newlemma2} we have
		\begin{eqnarray}\label{mainthmseceqn6}
	\sum_{t=0}^{\infty}g(t)\Bigl(\dfrac{1}{\sqrt{n}}\Bigr)^t&=&\sum_{t=0}^{2m}g(t)\Bigl(\dfrac{1}{\sqrt{n}}\Bigr)^t+\sum_{t=2m+1}^{\infty}g(t)\Bigl(\dfrac{1}{\sqrt{n}}\Bigr)^t\nonumber\\
	&=& \sum_{t=0}^{2m}g(t)\Bigl(\dfrac{1}{\sqrt{n}}\Bigr)^t+\sum_{t=m}^{\infty}g(2t+1)\Bigl(\dfrac{1}{\sqrt{n}}\Bigr)^{2t+1}+\sum_{t=m+1}^{\infty}g(2t)\Bigl(\dfrac{1}{\sqrt{n}}\Bigr)^{2t}\nonumber\\
	&=& \sum_{t=0}^{2m}g(t)\Bigl(\dfrac{1}{\sqrt{n}}\Bigr)^t+\sum_{t=m}^{\infty}(g_{o,1}(t)+g_{o,2}(t))\Bigl(\dfrac{1}{\sqrt{n}}\Bigr)^{2t+1}+\sum_{t=m+1}^{\infty}(g_{e,1}(t)+g_{e,2}(t))\Bigl(\dfrac{1}{\sqrt{n}}\Bigr)^{2t}.\nonumber\\
	\end{eqnarray}
	Using Lemmas \ref{errorsum1}-\ref{errorsum2} by substituting $k \mapsto m+1$ and Lemmas \ref{errorsum3}-\ref{errorsum4} by substituting $k \mapsto m$, it follows that
	\begin{eqnarray}\label{mainthmseceqn7}
	\dfrac{L_1(m+1)+L_2(m+1)}{\sqrt{n}^{2m+2}}+\dfrac{L_3(m)+L_4(m)}{\sqrt{n}^{2m+1}}&<&\sum_{t=2m+1}^{\infty}g(t)\Bigl(\dfrac{1}{\sqrt{n}}\Bigr)^t\nonumber\\
	&<&\dfrac{U_1(m+1)+U_2(m+1)}{\sqrt{n}^{2m+2}}+\dfrac{U_3(m)+U_4(m)}{\sqrt
		{n}^{2m+1}}.\nonumber\\
	\end{eqnarray}
	By Lemma \ref{errorlem5} with $k=2m+1$, 
	\begin{equation}\label{mainthmseceqn8}
	\dfrac{\sqrt{12}\ e^{\mu(n)}}{24n-1}\frac{1}{\mu(n)^{2m+1}}<\frac{e^{\pi\sqrt{2n/3}}}{4n\sqrt{3}}\frac{\widehat{U}_2(2m+1)}{\sqrt{n}^{2m+1}}.
	\end{equation}
	From \eqref{mainthmseceqn7} and \eqref{mainthmseceqn8} along with the fact that $U_1(m)+U_2(m)>0$, we obtain
	\begin{equation}\label{finalestim3}
	\begin{split}
	\dfrac{\sqrt{12}e^{\mu(n)}}{24n-1}\Biggl(1-\dfrac{1}{\mu(n)}+\dfrac{1}{\mu(n)^{2m+1}}\Biggr)
	&<\frac{e^{\pi\sqrt{2n/3}}}{4n\sqrt{3}} \Biggl(\sum_{t=0}^{2m}g(t)\Bigl(\dfrac{1}{\sqrt{n}}\Bigr)^t+\dfrac{\widehat{U}(2m+1)}{\sqrt{n}^{2m+1}}\Biggr)
	\end{split},
	\end{equation}
	where $$\widehat{U}(2m+1)=U_1(m+1)+U_2(m+1)+U_3(m)+U_4(m)+\widehat{U}_2(2m+1).$$
	Since for all $m \geq 0$, $L_1(m)+L_2(m)<0$, it follows that
	\begin{equation}\label{finalestim4}
	\begin{split}
	\dfrac{\sqrt{12}e^{\mu(n)}}{24n-1}\Biggl(1-\dfrac{1}{\mu(n)}-\dfrac{1}{\mu(n)^{2m+1}}\Biggr) 
	&>\frac{e^{\pi\sqrt{2n/3}}}{4n\sqrt{3}} \Biggl(\sum_{t=0}^{2m}g(t)\Bigl(\dfrac{1}{\sqrt{n}}\Bigr)^t+\dfrac{\widehat{L}(2m+1)}{\sqrt{n}^{2m+1}}\Biggr)
	\end{split}
	\end{equation}
	with $$\widehat{L}(2m+1)=L_1(m+1)+L_2(m+1)+L_3(m)+L_4(m)-\widehat{U}_2(2m+1).$$
	Next, we estimate $\widehat{U}_2(2m+1)$. Recall from Lemma \ref{errorlem5} that for all $n > n_0(2m+1)$,
	\begin{equation*}
	\widehat{U}_2(2m+1) < \dfrac{1}{\alpha^{2m+1}}\Bigl(1+\dfrac{2m+1}{3n}\Bigr) \dfrac{1}{\sqrt{24}^{2m+1}} < \dfrac{1}{\alpha}\Bigl(1+\dfrac{2m+1}{3n}\Bigr) \dfrac{1}{\sqrt{24}^{2m+1}}.
	\end{equation*}
	We note that for $0 \leq m \leq 10$; $n\geq 1>n_0(2m+1)$, and therefore, $\dfrac{1}{\alpha}\Bigl(1+\dfrac{2m+1}{3n}\Bigr)< \dfrac{8}{\alpha}$; whereas for $m \geq 11$, $n > \dfrac{2m+3}{24}$. This implies that $\dfrac{1}{\alpha}\Bigl(1+\dfrac{2m+1}{3n}\Bigr)< \dfrac{9}{\alpha}$. Hence, for all $n \geq \max \{1, n_0(2m+1)\}$,
	\begin{equation}\label{oddcasebound}
	\widehat{U}_2(2m+1) < \dfrac{9}{\alpha}\dfrac{1}{\sqrt{24}^{2m+1}}.
	\end{equation}
	From Lemmas \ref{errorsum1}-\ref{errorsum4} and \ref{errorlem5}, for all $n \geq \max \{1,n_0(2m+1)\}$, we get
	\begin{eqnarray}\label{oddcaseupperbound}
	\widehat{U}(2m+1)&<& \Biggl(18+\dfrac{5}{\sqrt{m+1}}+\dfrac{1}{(m+1)^{3/2}}+2\sqrt{m+2}\Biggr)\dfrac{1}{\sqrt{24}^{2m+1}} \nonumber\\
	&\leq & \dfrac{17\sqrt{m+2}}{\sqrt{24}^{2m+1}}=U(2m+1).
	\end{eqnarray}
	Similarly for all $n \geq \max \{1,n_0(2m+1)\}$, it follows that
	\begin{eqnarray}\label{oddcaselowerbound}
	\widehat{L}(2m+1)&>& \Biggl(-17-\dfrac{3}{\sqrt{m+1}}-\dfrac{1}{(m+1)^{3/2}}-13\sqrt{m+2}\Biggr)\dfrac{1}{\sqrt{24}^{2m+1}} \nonumber\\
	&\geq & -\dfrac{15\sqrt{m+2}}{\sqrt{24}^{2m+1}}=L(2m+1).
	\end{eqnarray}
	
Plugging \eqref{oddcaseupperbound} and \eqref{oddcaselowerbound} into \eqref{finalestim3} and \eqref{finalestim4}, respectively, and applying Theorem \ref{bprz1}, we get
	\begin{equation}\label{finalestimoddupper}
p(n)<	\dfrac{\sqrt{12}e^{\mu(n)}}{24n-1}\Biggl(1-\dfrac{1}{\mu(n)}+\dfrac{1}{\mu(n)^{2m+1}}\Biggr)
	<\frac{e^{\pi\sqrt{2n/3}}}{4n\sqrt{3}} \Biggl(\sum_{t=0}^{2m}g(t)\Bigl(\dfrac{1}{\sqrt{n}}\Bigr)^t+\dfrac{U(2m+1)}{\sqrt{n}^{2m+1}}\Biggr)
	\end{equation}
	and 
	\begin{equation}\label{finalestimoddlower}
p(n)>	\dfrac{\sqrt{12}e^{\mu(n)}}{24n-1}\Biggl(1-\dfrac{1}{\mu(n)}-\dfrac{1}{\mu(n)^{2m+1}}\Biggr)
	>\frac{e^{\pi\sqrt{2n/3}}}{4n\sqrt{3}} \Biggl(\sum_{t=0}^{2m}g(t)\Bigl(\dfrac{1}{\sqrt{n}}\Bigr)^t+\dfrac{L(2m+1)}{\sqrt{n}^{2m+1}}\Biggr).
	\end{equation}
\end{proof}

\begin{theorem}\label{Mainthm}
Let $\widehat{g}(k)$ be as in Theorem \ref{bprz1} and $g(t)$ as in \eqref{g(t)definition}. Let $L(w)$ and $U(w)$ be as in Definition \ref{newdef18}. If $w \in \mathbb{Z}_{\geq 1}$ with $\lceil w/2\rceil \geq 1$ and $n> \widehat{g}(w)$, then
\begin{equation}\label{Mainthmeqn}
\frac{e^{\pi\sqrt{2n/3}}}{4n\sqrt{3}} \Biggl(\sum_{t=0}^{w-1}g(t)\Bigl(\dfrac{1}{\sqrt{n}}\Bigr)^t+\dfrac{L(w)}{\sqrt{n}^{w}}\Biggr)< p(n) < \frac{e^{\pi\sqrt{2n/3}}}{4n\sqrt{3}} \Biggl(\sum_{t=0}^{w-1}g(t)\Bigl(\dfrac{1}{\sqrt{n}}\Bigr)^t+\dfrac{U(w)}{\sqrt{n}^{w}}\Biggr).
\end{equation}	
\end{theorem}
\begin{proof}
Combining Lemmas \ref{mainlemmaevencase} and \ref{mainlemmaoddcase} together with the fact that $\widehat{g}(k)>\max \bigl\{n_0(k),1\bigr\}$, we arrive at \eqref{Mainthmeqn}.	
\end{proof}

\begin{corollary}\label{cor1}
	For all $n \geq 116$, we have
	\begin{equation}\label{cor1eqn}
		\dfrac{e^{\pi\sqrt{2n/3}}}{4n\sqrt{3}}\Biggl(\sum_{t=0}^{3}\dfrac{g(t)}{\sqrt{n}^t}-\dfrac{1}{14 n^{2}}\Biggr)<p(n)<\dfrac{e^{\pi\sqrt{2n/3}}}{4n\sqrt{3}}\Biggl(\sum_{t=0}^{3}\dfrac{g(t)}{\sqrt{n}^t}+\dfrac{1}{13 n^{2}}\Biggr),
	\end{equation}
where 
\begin{equation*}
g(0)=1,\ g(1)=-\dfrac{\pi^2+72}{24\sqrt{6}\pi},\  g(2)=\dfrac{\pi^2+432}{6912},\  g(3)=-\dfrac{\pi^4+1296\pi^2+93312}{497664\sqrt{6}\pi}.
\end{equation*}
\end{corollary}
\begin{proof}
Plugging $w=4$ into \eqref{Mainthmeqn}, we obtain the inequality \eqref{cor1eqn}.
\end{proof}
\begin{remark}\label{corremark1}
Corollary \ref{cor1} provides an answer to the Question \ref{Chenproblem}, asked by Chen. As a consequence from \eqref{cor1eqn}, one can derive that $p(n)$ is $\log$-concave for all $n \geq 26$.
\end{remark}

\section{Appendix}\label{Appendixp(n)asymp}
\subsection{Proofs of the lemmas presented in Section \ref{sect2}.}\label{proofoflemmas}\hfill\\
\emph{Proof of Lemma \ref{lem1}:}
	For $n=1$ we have to prove
	$$\frac{1-x_1}{1+y_1}\geq 1-x_1-y_1,$$
	which is equivalent to
	$$1-x_1\geq (1+y_1)(1-x_1-y_1)\geq 1-x_1-y_1^2-x_1y_1 \Leftrightarrow 0\geq -y_1^2-x_1y_1.$$
	This is always true because $x_1,y_1$ are non-negative real numbers.
	Now assume by induction that the statement is true for $n=N$. Next we prove the statement for $n=N+1$. 
	For $n=N$, we have $P\geq (1-S)$
	with $$P:=\frac{(1-x_1)(1-x_2)\cdots (1-x_N)}{(1+y_1)(1+y_2)\cdots(1+y_N)}\ \ \text{and}\ \ S:=\sum_{j=1}^Nx_j+\sum_{j=1}^Ny_j.$$
	This implies that
	$$P\frac{1-x_{N+1}}{1+y_{N+1}}\geq (1-S)\frac{1-x_{N+1}}{1+y_{N+1}}.$$
	Therefore it suffices to prove that
	$$(1-S)\frac{1-x_{N+1}}{1+y_{N+1}}\geq 1-S-y_{N+1}-x_{N+1},$$
	which is equivalent to
	$$(1-S)(1-x_{N+1})\geq (1-S-y_{N+1}-x_{N+1})(1+y_{N+1}).$$
	Equivalently,
	$$1-x_{N+1}-S+S\cdot x_{N+1}\geq 1-S-x_{N+1}-y^2_{N+1}-x_{N+1}y_{N+1}-S\cdot y_{N+1},$$
	which amounts to say that
	$$S\cdot x_{N+1}\geq -y^2_{N+1}-x_{N+1}y_{N+1}-S\cdot y_{N+1},$$
	and this inequality holds 
	because $x_{N+1},y_{N+1},S\geq 0$.
\qed 

\emph{Proof of Lemma \ref{lem2}:}
	Expanding the quotient $\frac{(-1)^i(-t)_i}{(t)_i}$ as 
	$$(-1)^i\frac{(-t)_i}{(t)_i}=(-1)^i\prod_{j=1}^i\frac{-t+j-1}{t+j-1}=\prod_{j=1}^i\frac{t-j+1}{t+j-1},$$
	we obtain
	\begin{equation*}
	\frac{t(-t)_u(-1)^u}{(1+2t)(t+u)(t)_u}=\frac{t}{2(t+\frac{1}{2})(t+u)}\prod_{j=1}^u\frac{t-(j-1)}{t+j-1}=\frac{1}{2t(1+\frac{1}{2t})(1+\frac{u}{t})}\prod_{j=1}^u\frac{1-\frac{j-1}{t}}{1+\frac{j-1}{t}}.
	\end{equation*}
	Since $t \geq 1$ and $u<t$, it is clear that
	\begin{equation}\label{lem2eqn1}
	\frac{1}{2t(1+\frac{1}{2t})(1+\frac{u}{t})}\prod_{j=1}^u\frac{1-\frac{j-1}{t}}{1+\frac{j-1}{t}}\leq \dfrac{1}{2t}.
	\end{equation}
	By Lemma \ref{lem1}, it follows that
	\begin{equation}\label{lem2eqn2}
	\frac{1}{2t(1+\frac{1}{2t})(1+\frac{u}{t})}\prod_{j=1}^u\frac{1-\frac{j-1}{t}}{1+\frac{j-1}{t}} \geq  \frac{1}{2t}\Bigl(1-\frac{\frac{1}{2}+u+2\sum_{j=1}^u (j-1)}{t}\Bigr)= \frac{1}{2t}\Bigl(1-\frac{u^2+\frac{1}{2}}{t}\Bigr).
	\end{equation}
	Combining \eqref{lem2eqn1} and \eqref{lem2eqn2} concludes the proof.
\qed 

\emph{Proof of Lemma \ref{lem3}:}
	By Lemma \ref{lem1}, 
	\begin{equation}\label{lem3eqn1}
	\frac{1}{2t}\geq \frac{1}{1+2t}=\frac{1}{2t}\frac{1}{(1+\frac{1}{2t})}\geq\frac{1}{2t}\Bigl(1-\frac{1}{2t}\Bigr)\geq \frac{1}{2t}-\frac{1}{4t^2}. 
	\end{equation}
	Now
	\begin{equation*}
	\frac{2t\sum_{i=1}^u\frac{(-t)_i(-1)^i}{(t+i)(t)_i}}{1+2t}=\frac{1}{1+\frac{1}{2t}}\sum_{i=1}^u\frac{1}{t+i}\prod_{j=1}^i\frac{t-j+1}{t+j-1}=\frac{1}{t}\frac{1}{1+\frac{1}{2t}}\sum_{i=1}^u\frac{1}{1+\frac{i}{t}}\prod_{j=1}^i\frac{1-\frac{j-1}{t}}{1+\frac{j-1}{t}}. 
	\end{equation*}
	As $t \geq 1$ and $u<t$, it directly follows that
	\begin{equation}\label{lem3eqn2}
	\frac{1}{t}\frac{1}{1+\frac{1}{2t}}\sum_{i=1}^u\frac{1}{1+\frac{i}{t}}\prod_{j=1}^i\frac{1-\frac{j-1}{t}}{1+\frac{j-1}{t}} \leq \dfrac{u}{t}.
	\end{equation}
	Applying Lemma \ref{lem1}, we obtain
	\begin{equation}\label{lem3eqn3}
	\frac{1}{t}\frac{1}{1+\frac{1}{2t}}\sum_{i=1}^u\frac{1}{1+\frac{i}{t}}\prod_{j=1}^i\frac{1-\frac{j-1}{t}}{1+\frac{j-1}{t}} \geq \frac{1}{t}\sum_{i=1}^u1-\frac{\frac{1}{2}+i+2\sum_{j=1}^i(j-1)}{t}
	=\frac{u}{t}-\frac{u(2u^2+3u+4)}{6t^2}.
	\end{equation}
	Finally, \eqref{lem3eqn1}, \eqref{lem3eqn2}, and \eqref{lem3eqn3} imply the desired inequality. 
\qed

\emph{Proof of Lemma \ref{mainlemma2}:}
	Let $n\geq u$ be fixed. We have to show that $b_n\geq a_n$. First we note that
	$$a_{k+1}-a_n=\sum_{j=n}^k (a_{j+1}-a_j)\geq \sum_{j=n}^k (b_{j+1}-b_j)=b_{k+1}-b_n.$$
	Consequently, for all $k\geq n$ we have
	$$a_{k+1}-a_n\geq b_{k+1}-b_n\Leftrightarrow b_n-b_{k+1}\geq a_n-a_{k+1}.$$
	This implies that
	$$b_n=\lim_{k\to \infty} (b_n-b_{k+1})\geq \lim_{k\to \infty}(a_n-a_{k+1})=a_n.$$
\qed

\emph{Proof of Lemma \ref{helperC}:}
	We apply Lemma \ref{mainlemma2} with $a_n=\sum_{u=n+1}^{\infty} \frac{u^k\alpha^{2u}}{(2u)!}$ and $b_n=\frac{C_k}{n^2}$:
	$$a_{n+1}-a_n=-\frac{(n+1)^k\alpha^{2n+2}}{(2n+2)!}\ \ \text{and}\ \ b_{n+1}-b_n=-\frac{C_k(2n+1)}{n^2(n+1)^2}.$$ 
	Therefore $b_{n+1}-b_n\leq a_{n+1}-a_n$ is equivalent to
	$$\frac{(n+1)^k\alpha^{2n+2}}{(2n+2)!}\leq \frac{C_k(2n+1)}{n^2(n+1)^2}\Leftrightarrow f(n):=\frac{n^2(n+1)^{k+2}\alpha^{2n+2}}{(2n+1)(2n+2)!}\leq C_k .$$
	In order to prove $f(n)\leq C_k$, it suffices to prove $f(m)\leq C_k$, where $m$ is such that $f(m)$ is maximal. Hence in order to find such a $m$, we find the first $m$ such that $f(m+1)\leq f(m)$. This is equivalent to finding $\frac{f(m+1)}{f(m)}\leq 1$, also as we will see there is only one such maximum. Then $\underset{n\in\mathbb{N}}{\max}\ f(n)=f(m)$.
	Now
	$$\frac{f(n+1)}{f(n)}=\frac{(n+1)^2(n+2)^{k+2}\alpha^{2n+4}}{(2n+3)(2n+4)!}\frac{(2n+1)(2n+2)!}{n^2(n+1)^{k+2}\alpha^{2n+2}}=\frac{\alpha^2(n+2)^{k+2}(2n+1)}{(2n+4)(2n+3)^2(n+1)^kn^2}.$$
	Using Mathematica's implementation of Cylindrical Algebraic Decomposition \cite{Collins1975},  we obtain that $$\frac{\alpha^2(n+2)^{k+2}(2n+1)}{(2n+4)(2n+3)^2(n+1)^kn^2}\leq 1, \text{ for all $\alpha^2\leq \frac{800}{729}$.}$$ As $\alpha^2=\frac{\pi^2}{36}<\frac{800}{729}$,  $\underset{n\in\mathbb{N}}{\max}\ f(n)=f(1)$; i.e., $f(n)\leq f(1)=C_k$. 
	
\qed

\subsection{The Sigma simplification of $S_3(t,u)$ in Lemma \ref{S3lemma}}\label{reductionbySigma}\hfill\\
Using the symbolic summation package \texttt{Sigma}~\cite{Schneider:07a} and its underlying machinery in the setting of difference rings~\cite{Schneider:21} the inner sum $S_3(t,u)$ 
can be simplified as follows.  Recall from \eqref{S3lemprfeqn1} that
$$S_3(t,u)
=\sum_{s=0}^{t-u}\frac{1}{s+u}\Bigl(\frac{1}{2}-s-u\Bigr)_{s+u+1}\binom{-\frac{3}{2}}{t-s-u}\frac{(-s-u)_u}{(s+2u)!}.$$
After loading \texttt{Sigma} into the computer algebra system Mathematica
\begin{mma}
	\In << Sigma.m \\
	\vspace*{-0.1cm}
	\Print \LoadP{Sigma - A summation package by Carsten Schneider
		\copyright\ RISC-JKU}\\
\end{mma}

\noindent we input the sum under consideration

\begin{mma}
	\In mySum3=\sum_{s=0}^{t-u}\frac{1}{s+u}\Bigl(\frac{1}{2}-s-u\Bigr)_{s+u+1}\binom{-\frac{3}{2}}{t-s-u}\frac{(-s-u)_u}{(s+2u)!};\\
\end{mma}

\noindent and compute a recurrence of it by executing

\begin{mma}\MLabel{MMA:rec}
	\In rec3=GenerateRecurrence[mySum3]\\
	\Out (t - u) u SUM[u] + 2 (2 + t) (1 + u) SUM[u + 1] + (2 + u) (2 + t + u) SUM[u + 2] == 0\\
\end{mma}

\noindent As a result we get a homogeneous linear recurrence of order $2$ for $S_3(t,u)=\texttt{SUM[u](=mySum3)}$. Internally, Zeilberger's creative telescoping paradigm~\cite{AequalB} is applied which not only provides a recurrence but delivers simultaneously a proof certificate that guarantees the correctness of the result.

\medskip

\noindent\textit{Verification of the recurrence.} Denote the summand of $S_3(t,u)$ by $f(t,u,s)$; i.e. set
$$f(t,u,s)=\frac{1}{s+u}\Bigl(\frac{1}{2}-s-u\Bigr)_{s+u+1}\binom{-\frac{3}{2}}{t-s-u}\frac{(-s-u)_u}{(s+2u)!}.$$
Then one can verify that the polynomials $a_0(t,u)=u (t-u)$, $a_1(t,u)=2 (2+t) (1+u)$ and $a_2(t,u)=(2+u) (2+t+u)$ (free of the summation variable $s$) and the expression
$$g(t,u,s)=-\frac{\gamma(t,u,s) s \binom{-\frac{3}{2}}{-s
		+t
		-u
	} (-s
	-u
	)_u \big(
	\frac{1}{2}
	-s
	-u
	\big)_{1
		+s
		+u
}}{(s
	+2 u
	)!(s
	+u
	) (1
	+s
	+2 u
	) (2
	+s
	+2 u
	) (3
	+s
	+2 u
	) (-1
	+2 s
	-2 t
	+2 u
	)}$$
with 
\begin{align*}
\gamma(t,u,s)=&-6 s
-6 s^2
+16 s^3
+8 s^4
+6 t
-6 s t
-46 s^2 t
-20 s^3 t
+12 t^2
+30 s t^2
+12 s^2 t^2\\
&-12 u
-22 s u
+66 s^2 u
+64 s^3 u
+8 s^4 u
-7 t u
-138 s t u
-126 s^2 t u
-16 s^3 t u\\
&+52 t^2 u
+57 s t^2 u
+8 s^2 t^2 u
-27 u^2
+88 s u^2
+172 s^2 u^2
+44 s^3 u^2
-108 t u^2\\
&-235 s t u^2
-68 s^2 t u^2
+57 t^2 u^2
+24 s t^2 u^2
+32 u^3
+192 s u^3
+92 s^2 u^3\\
&-140 t u^3
-98 s t u^3
+18 t^2 u^3
+75 u^4
+86 s u^4
-48 t u^4
+30 u^5
\end{align*}
satisfy the summand recurrence
\begin{equation}\label{Equ:SummandRec}
g(t,u,s+1)-g(t,u,s)=a_0(t,u) f(t,u,s)+a_1(t,u) f(t,u+1,s)+a_2(t,u) f(t,u+2,s)
\end{equation}
for all $0\leq s\leq t-u$ with $t\geq u$. The components of the summand recurrence can be obtained with the function call \texttt{CreativeTelescoping[mySum3]}.
Summing the verified equation~\eqref{Equ:SummandRec} over $s$ from $0$ to $t-u$ yields the output recurrence~\myOut{\ref{MMA:rec}}, which at the same time yields a proof for the correctness of \myOut{\ref{MMA:rec}}.\qed

\medskip

We remark that \texttt{Sigma}'s creative telescoping approach works not only for hypergeometric sums (here one could use, for instance, also the Paule-Schorn implementation ~\cite{PauleSchorn:95} of Zeilberger's algorithm~\cite{AequalB}), but can be applied in the general setting of difference rings which allows to treat summands built by indefinite nested sums and products. More involved examples in the context of plane partitions can be found, e.g., in~\cite{APS:05}. 

We are now in the position to solve the output reccurrence \myOut{\ref{MMA:rec}} with the function call

\begin{mma}
	\In recSol=SolveRecurrence[rec3,SUM[u]]\\
	\Out \bigg\{\big\{	0 , (-1)^u\big\},
	\big\{0 ,\frac{(2
		+t
		-u
		)}{u (2
		+t
		+u
		)}\frac{(-t)_u}{(2+t)_u}+2(-1)^u 
	\sum_{i=1}^u \frac{(-1)^i (-t)_i}{(2
		+i
		+t
		) (2+t)_i}\big\}, \big\{1,0\big\}\bigg\}\\
\end{mma}

\noindent This means that we found two linearly independent solutions (the list entries whose first entry is a zero) that span the full solution space, i.e., the general solution to \myOut{\ref{MMA:rec}} is
\begin{equation}\label{Equ:GeneralSol}
G(t,u)=c_1(t)\,(-1)^u+c_2\Big(\frac{(2
	+t
	-u
	) }{u (2
	+t
	+u
	)}\frac{(-t)_u}{(2+t)_u}+2 (-1)^u 
\sum_{i=1}^u \frac{(-1)^i (-t)_i}{(2
	+i
	+t
	) (2+t)_i}
\Big)
\end{equation}
where the $c_1,c_2$ are constants being free of $u$.
For further details on the underlying machinery (inspired by~\cite{AequalB}) we refer to~\cite{Schneider:21}. 

\medskip

\noindent\textit{Verification of the general solution}.
The correctness of the solutions can be verified by plugging them into the recurrence~\myOut{\ref{MMA:rec}} and applying (iteratively) the shift relations
\begin{align*}
(-1)^{u+1}&=-(-1)^u,\\
(-t)_{1+u}&=(-t
+u
) (-t)_u,\\
(2+t)_{1+u}&=(2
+t
+u
) (2+t)_u,\\
\sum_{i=1}^{1+u} \frac{(-1)^i (-t)_i}{(2
	+i
	+t
	) (2+t)_i}&=
\sum_{i=1}^u \frac{(-1)^i (-t)_i}{(2
	+i
	+t
	) (2+t)_i}
+\frac{(-1)^u (t
	-u
	) (-t)_u}{(2
	+t
	+u
	) (3
	+t
	+u
	) (2+t)_u}.
\end{align*}
Then simple rational function arithmetic shows that the obtained expression collapses to zero.\qed

\medskip

Finally, we compute the first two initial values (by another round of symbolic summation) and find that
\begin{equation}\label{Equ:InitialV}
\begin{split}
S_3(t,1)&=(-1)^{t}-\frac{(t+2)}{2 (1+t)}\binom{-\frac{3}{2}}{t},\\
S_3(t,2)&=-(-1)^{t}
+\frac{\big(
	8+7 t+t^2\big)}{4 (1+t) (2+t)}\binom{-\frac{3}{2}}{t}.
\end{split}
\end{equation}
With this information we can set  
$c_1=-(-1)^{t}+\frac{(3 t+4)}{2 (t+1) (t+2)}\binom{-\frac{3}{2}}{t}$ and $c_2=\frac{1}{2} \binom{-\frac{3}{2}}{t}$ so that the general solution~\eqref{Equ:GeneralSol} agrees with $S_3(t,u)$ for $u=1,2$. Since $S_3(t,u)$ and the specialized general solution are both solutions of the recurrence~\myOut{\ref{MMA:rec}} and the first two initial values agree, they are identical for all $u\geq0$ with $u\leq t$. This last step of combining the solutions accordingly can be accomplished by inserting the list of two initial values
\begin{mma}
	\In initialL=\bigg\{(-1)^t
	-\frac{(t+2)}{2 (1+t)}\binom{-\frac{3}{2}}{t},
	-(-1)^{t}+\frac{\big(
		8+7 t+t^2\big)}{4 (1+t) (2+t)}\binom{-\frac{3}{2}}{t}
	\bigg\};\\
\end{mma}
\noindent and then executing the command
\begin{mma}
	\In FindLinearCombination[recSol3,\{1,initialL\},u,2]\\
	\Out -(-1)^{t}(-1)^{u}
	+\frac{1}{2} \binom{-\frac{3}{2}}{t}\Big(
	\frac{(2
		+t
		-u
		)}{u (2
		+t
		+u
		)}\frac{(-t)_u}{(2+t)_u}+(-1)^u\Big(
	\frac{1}{1+t}
	+\frac{2}{2+t}
	+2 
	\sum_{i=1}^u \frac{(-1)^i (-t)_i}{(2
		+i
		+t
		) (2+t)_i}
	\Big) 
	\Big)\\
\end{mma}

Carrying out all the steps above (including also the calculation of the initial values) can be rather cumbersome. In order to support the user with the simplification of such problems, the package 

\begin{mma}
	\In << EvaluateMultiSums.m \\
	\vspace*{-0.1cm}
	\Print \LoadP{EvaluateMultiSum by Carsten Schneider
		\copyright\ RISC-JKU}\\
\end{mma}

\noindent has been developed. More precisely, by applying the command \texttt{EvaluateMultiSums} to the input sum $\texttt{mySum3}(=S_3(t,u))$, all the above steps are carried out automatically and one obtains in one stroke the desired result:

\begin{mma}
	\In sol3=EvaluateMultiSum[mySum3,\{\},\{u, t\}, \{0, 1\}, \{t, \infty\}]\\
	\Out -(-1)^{t}(-1)^{u}
	+\frac{1}{2} \binom{-\frac{3}{2}}{t}\Big(
	\frac{(2
		+t
		-u
		)}{u (2
		+t
		+u
		)}\frac{(-t)_u}{(2+t)_u}+(-1)^u\Big(
	\frac{1}{1+t}
	+\frac{2}{2+t}
	+2 
	\sum_{i=1}^u \frac{(-1)^i (-t)_i}{(2
		+i
		+t
		) (2+t)_i}
	\Big) 
	\Big)\\
\end{mma}

Since we prefer to rewrite the found expression in terms of the Pochhammer symbol $(t)_u$, we execute the final simplification step with the function call

\begin{mma}
	\In SigmaReduce[sol3, u, Tower\to \{(t)_u\}]\\
	\Out -(-1)^{t}(-1)^{u}
	+\binom{-\frac{3}{2}}{t}\Big(
	\frac{t (1
		+2 t
		-2 u
		)}{2 (1+2 t) u (t
		+u
		)}\frac{(-t)_u}{(t)_u}+
	(-1)^u\Big(\frac{1}{1+2 t}
	+\frac{2 t}{1+2 t}\sum_{i=1}^u \frac{(-1)^i (-t)_i}{(i
		+t
		) (t)_i}\Big)
	\Big)\\
\end{mma}

\begin{remark}\label{remarkappendix}
We should mention that there is no particular reason for explaining the details of \textnormal{\texttt{Sigma}} application only for $S_3(t,u)$. The simplification of the sums $S_1(t,u)$, $S_2(t,u)$, and $S_4(t,u)$, as in \eqref{S1lemprfeqn1}, \eqref{S2lemprfeqn1}, and \eqref{S4lemprfeqn1}, respectively, works completely analogously.  
\end{remark}

\section{Concluding remarks}\label{outlook}
We conclude this paper with a list of possible future work based on the method devised in this paper and its further applications. 
\begin{enumerate}
	\item A prudent application of our method might lead to obtaining full asymptotic expansion and respective error bounds for a broad class of functions; for example: $q(n)$-partitions into distinct parts, $p^{s}(n)$-partitions into perfect $s$th powers, $k$-colored partitions, $k$-regular partitions, Andrews' spt-function, $\alpha(n)$-$n$th coefficient of Ramanujan's third order mock theta function $f(q)$, the coefficient sequence of Klein's $j$-function, etc.  
	\item More generally consider the class of Dedekind $\eta$-quotients which fit perfectly into \cite[Thm. 1.1]{Chern} or \cite[Thm. 1.1]{Sussman}. Therefore one can also obtain a full asymptotic expansion and infinite families of inequalities for the coefficient sequence arising from the Fourier expansion of the considered Dedekind $\eta$-function.
	\item Theorem \ref{Mainthm} can be utilized as a black box in order to prove inequalities pertaining to the partition function by constructing an unified framework. A major class of inequalities for $p(n)$ can be separated into the following two categories among many others:
	\begin{enumerate}
		\item Tur\'{a}n inequalities and its higher order analogues related to the real rootedness of Jensen polynomials associated to $p(n)$, studied in \cite{desalvo2015log}, \cite{chen2019higher}, and \cite{griffin2019jensen}.
		\item Linear homogeneous inequalities for $p(n)$; i.e.,
		$$\sum_{i=1}^{r}p(n+x_i)\leq \sum_{i=1}^{s}p(n+y_i).$$
		For more details we refer to \cite{katriel2018asymptotically,merca2020general}.
	\end{enumerate}
\item More generally, it would be interesting to design a constructive method to decide whether for some positive integer $N$ a relation of the form
$$\sum_{j=1}^{M_1}\prod_{i=1}^{T_1}p(n+s^{(j)}_i) \leq \sum_{j=1}^{M_2}\prod_{i=1}^{T_2}p(n+r^{(j)}_i)$$
holds for all $n \geq N$ or not.
\end{enumerate}

\begin{center}
	\textbf{Acknowledgements}
\end{center}
Banerjee was funded by the Austrian Science Fund (FWF): W1214-N15, project
DK6. Paule and Radu were supported by grant SFB F50-06 of the FWF. Schneider has received funding from the Austrian Science Fund (FWF) grants SFB F50 (F5009-N15) and P33530.

\section{Conflict of interest statement}
On behalf of all authors, the corresponding author states that there is no conflict of interest.
\section{Data availability statement}
No datasets were generated or analysed during the current study.


\end{document}